\documentclass[a4paper,reqno,11pt]{amsart}

\usepackage[margin=1in]{geometry}
\usepackage{amsmath, amssymb,amsthm,url,mathrsfs,graphicx,amscd,amsfonts}
\usepackage[mathscr]{eucal}
\usepackage{dsfont}
\usepackage{mathtools}
\usepackage[all]{xy}

\usepackage[utf8]{inputenc}
\usepackage{graphicx,todonotes,hyperref}
\usepackage{epstopdf}
\usepackage{inputenc}
\usepackage{enumitem}
\usepackage{amsmath}
\usepackage{graphicx}
\usepackage{xcolor}
\usepackage{xspace}
\usepackage{amssymb}
\usepackage{amsfonts}
\usepackage{amssymb}
\usepackage{graphicx}
\usepackage{mathtools}
\usepackage{amsthm}
\usepackage{tikz-cd}
\usepackage{multirow}
\usepackage{enumitem}

\input xy
\xyoption{all}

\newtheorem*{acknowledgements*}{Acknowledgements}

\newcommand{\s}{\mathbb{S}}
\newcommand{\z}{\mathbb{Z}}

\newcommand{\Z}{\mathbb{Z}}
\newcommand{\C}{\mathbb{C}}

\newcommand{\Hom}{\mbox{Hom}}
\newcommand{\Ext}{\mbox{Ext}}

\newcommand{\rel}{\operatorname{rel}}

\allowdisplaybreaks
\theoremstyle{definition}
\newtheorem{theorem}{Theorem}[section]
\newtheorem{prop}[theorem]{Proposition}
\newtheorem{lemma}[theorem]{Lemma}
\newtheorem{cor}[theorem]{Corollary}
\newtheorem{note}[theorem]{Note}
\newtheorem{defn}[theorem]{Definition}

\newtheorem{rmk}[theorem]{Remark}
\newtheorem{fact}[theorem]{Fact}

\newtheorem{thm}[theorem]{Theorem}
\newtheorem*{theorem*}{Theorem}
\newtheorem*{corollary*}{Corollary}
\newtheorem*{prop*}{Proposition}
\newtheorem*{rmk*}{Remark}
\newtheorem{manualtheoreminner}{Theorem}
\newenvironment{manualtheorem}[1]{%
	\renewcommand\themanualtheoreminner{#1}%
	\manualtheoreminner
}{\endmanualtheoreminner}

\begin{document}
	\title[Smooth Structures on $M\times\mathbb{S}^k$]{Smooth Structures on $M\times\mathbb{S}^k$}
	
	\author{Samik Basu}
	\address{Stat Math Unit, Indian Statistical Institute, Kolkata-700108, India}
	\email{samikbasu@isical.ac.in}
	
	\author{Ramesh Kasilingam}
	\address{Department of Mathematics, Indian Institute of Technology Madras, Chennai-600036, India}
	\email{rameshkasilingam.iitb@gmail.com  ; rameshk@iitm.ac.in}
	\author{Ankur Sarkar}
	\address{The Institute of Mathematical Sciences, A CI of Homi Bhabha National Institute, Chennai-600113, India}
	\email{ankurs@imsc.res.in; ankurimsc@gmail.com}	
	

	\begin{abstract}
This paper explores various differentiable structures on the product manifold $M \times \mathbb{S}^k$, where $M$ is either a 4-dimensional closed oriented manifold or a simply connected 5-dimensional closed manifold. We identify the possible stable homotopy types of $M$ and use it to calculate the concordance inertia group and the concordance structure set of $M\times\mathbb{S}^k$ for $1\leq k\leq 10.$ These calculations enable us to further classify all manifolds that are homeomorphic to $\mathbb{C}P^2\times\mathbb{S}^k$, up to diffeomorphism, for each $4\leq k\leq 6$.
     \end{abstract}
	\maketitle

	\section{Introduction} 
 We work in the smooth category of closed, connected, oriented smooth manifolds and orientation preserving maps. One of the main interests in the smooth category is the classification of all smooth manifolds up to diffeomorphism. The discovery of exotic spheres by Milnor in 1956 \cite{milnor} has sparked global interest and intensive research into the existence and non-existence of exotic smooth structures. In this paper, we turn our attention to the investigation of exotic smooth structures on the product manifold $M\times\mathbb{S}^k,$ where $M$ is either a $4$-dimensional manifold or a simply connected $5$-dimensional manifold with $1\leq k\leq 10.$  Focusing on product manifolds, extensive studies have been conducted to classify smooth structures on $\mathbb{S}^p\times\mathbb{S}^q$ \cite{kubo, Kawakubo69, schultz, DC10, XW21}, $\mathbb{C}P^2\times\mathbb{S}^7$ \cite{bel}, and $\mathbb{C}P^2\times\mathbb{S}^3$ \cite{XW22}. The inertia groups of these manifolds have been identified as key elements in these studies.
 
 Recall that the \textit{inertia group} $I(M)$ of a closed smooth oriented $n$-manifold $M$ is the subgroup of the Kervaire-Milnor group $\Theta_n$ \cite{milnor} consisting of homotopy $n$-spheres $\Sigma$ such that $M\#\Sigma$ is diffeomorphic to $M$ via an orientation-preserving diffeomorphism. Generally, there are no systematic methods available for determining the inertia group of a manifold. Therefore, one might consider studying a specific subgroup of $I(M)$, referred to as the \textit{concordance inertia group} of $M$ which is defined as follows: 
	 \begin{defn}
	     Let $M$ be a smooth $n$-manifold. Consider the pair $(N,f)$, where $N$ is a smooth manifold and $f: N \to M$ is a homeomorphism. We say that two pairs $(N_1,f_1)$ and $(N_2,f_2)$ are \textit{concordant} if there exist a diffeomorphism $\Phi: N_1\to N_2$ and a homeomorphism $H: N_1\times [0,1]\to M\times [0,1]$ such that $H\vert_{N_1\times\{0\}}=f_1$ and $H\vert_{N_1\times \{1\}}=f_2\circ \Phi.$ We denote the set of all such concordance classes by $\mathcal{C}(M)$ and represent the concordance class of $(N,f)$ in $\mathcal{C}(M)$ as $[(N,f)]$. The base point of $\mathcal{C}(M)$ is the concordance class $[(M,Id_{M})]$ of the identity map $Id_{M} : M\to M$.
      
        The \textit{concordance inertia group} $I_c(M)$ of $M$ is the subgroup of $I(M)$ consisting of homotopy $n$-spheres $\Sigma$ such that $(M,Id_{M})$ and $(M\#\Sigma, h_{\Sigma})$ are concordant, where $h_{\Sigma}:M\#\Sigma\to M$ is the standard homeomorphism.
     \end{defn}
Following Kirby and Siebenmann \cite[Page 194]{kirby}, we have that there exists a connected $H$-space $Top/O$ such that there is a bijection between $\mathcal{C}(N)$ and $[N, Top/O]$ for any closed smooth $n$-manifold $N$ with $n\geq 5$. Furthermore, the concordance class of $(N, Id_{N})$ corresponds to the homotopy class of the constant map under this bijection. 
   Additionally, since $Top/O$ is an infinite loop space  \cite{Boardmanvogt}, it enables the computation of $I_c(N)$ and $\mathcal{C}(N)$ by establishing the stable homotopy type of $N$. 
  In particular, if $N=M\times\mathbb{S}^k,$ where $M$ is either a $4$-dimensional manifold or a simply connected $5$-dimensional manifold, the concordance inertia group $I_c(N)$ is explicitly computed using the stable decomposition of the manifold $M$. This decomposition is achieved through certain cohomology operations (see Section \ref{stabletypeM^4} and \ref{stabletypeM^5}). To be more precise, we prove the following :

\begin{manualtheorem}{A} \label{Theorem_A}\noindent
 Let $M$ be a closed, smooth, oriented $n$-manifold. The following statements hold.
		\begin{enumerate}			
			\item Suppose $M$ is a non-spin $4$-manifold. Then
   \begin{enumerate}
   \item For $3\leq k\leq 16,$  we have 
                 \[I_c(M\times\mathbb{S}^k)= \begin{cases}
                     \z/2, & \mbox{if $k=5,6,7,11,$ or $13$},\\
                     \z/2\oplus\z/2, & \mbox{if $k=14$}, \\
                     0 & \mbox{otherwise}.
                 \end{cases}\]
    \item For $k=17,18, 8n-2~(n\geq 3)$,~ $\z/2\subseteq I_c(M\times \mathbb{S}^k)$. 
		\end{enumerate}
			\item If $M$ is a simply connected non-spin $5$-manifold, then
            \begin{enumerate}
                 \item $I_c(M\times\mathbb{S}^k)=0$ for all $2\leq k\leq 10,$ excluding $k=4,5,6,$ and $10.$ 
                \item For $k=4,5,6,$ and $10,$
                \begin{enumerate}
                    \item   If a least positive integer $r$ exists such that $\mathbb{Z}/{2^r} \subseteq H_2(M;\mathbb{Z}),$ then:
                      \begin{enumerate}
				\item  \[ I_c(M\times\s^4)= \begin{cases} \Z/2\oplus\z/2,  &\mbox{if $r=1$ or $2$,}\\ \Z/2, &\mbox{if $r\geq 3$,}  \end{cases} \] 
				\item $I_c(M\times\s^k)=\mathbb{Z}/2$ for $k= 5, 10$,
				\item  \[ I_c(M\times\mathbb{S}^6)=\begin{cases} 
					\z/4, & \mbox{if $r=1$},\\
					\z/2, &\mbox{if $r\geq 2$}.
				\end{cases}\]
                 
			\end{enumerate}
                  \item If $H_2(M;\mathbb{Z})$ does not contain $2$-torsion, then $I_c(M\times\mathbb{S}^k)=\mathbb{Z}/2.$
                \end{enumerate}
              
               \end{enumerate}

            \item If $M$ is either a spin $4$-manifold or a simply connected spin $5$ manifold, then $I_c(M\times\mathbb{S}^k)=0$ for all $k\geq 1.$
		\end{enumerate}
  (Theorem \ref{Theorem_A} is a synthesis of Theorem \ref{concorM^4}, Proposition \ref{inerM^4} and Theorem \ref{inerM^5}).
	\end{manualtheorem}
    Note that when  \( \dim(M \times \mathbb{S}^k)=n+k = 5 \) or \( 6 \), the concordance inertia group \( I_c(M \times \mathbb{S}^k) \) is trivial, since \( \Theta_{n+k} = 0 \) for \( n+k = 5,6 \) \cite{milnor}.

The techniques involved in the proof of Theorem A are explicit calculations that make use of the non-triviality of certain elements in the stable homotopy groups. Following these calculations, we also determine the concordance smooth structure set of $M\times\mathbb{S}^k$ based on the stable homotopy type of aforementioned manifolds $M$ (refer to Propositions \ref{cM^4} and \ref{cM^5}). In particular, we apply these computations to the manifolds $\mathbb{C}P^2\times\mathbb{S}^k,$ where $4\leq k\leq 6$, and then identify the orbits of elements of $\mathcal{C}(\mathbb{C}P^2\times\mathbb{S}^k)$ under the action of the group of homotopy classes of self-homotopy equivalences. As a result of these computations, we are able to establish the following classification result.

\begin{manualtheorem}{B}\label{theo}
    Suppose $N$ is a closed smooth oriented manifold that is homeomorphic to $\mathbb{C}P^2\times\mathbb{S}^k.$ Then 
    \begin{itemize}
        \item[(a)] For $k=4$ and $5,$ $N$ is oriented diffeomorphic to $\mathbb{C}P^2\times\mathbb{S}^k.$
        \item[(b)] When $k=6,$ the manifold $N$ is oriented diffeomorphic to exactly one of the manifolds $\mathbb{C}P^2\times\mathbb{S}^6, (\mathbb{C}P^2\times\mathbb{S}^6)\#\Sigma_{\alpha_1},$ or $(\mathbb{C}P^2\times\mathbb{S}^6)\#(\Sigma_{\alpha_1})^{-1},$ where $\Sigma_{\alpha_1}\in \Theta_{10}$ is the exotic sphere of order $3$.
    \end{itemize}
    (Theorem \ref{theo} is a combination of Theorems \ref{classificationCP2S4}, \ref{classificationCP2S5} and \ref{classificationCP2S6}.)
\end{manualtheorem}

Note that it follows from \cite[Corollary 4.2]{MMRS} that $I(\mathbb{C}P^2\times\mathbb{S}^3)=\Theta_7$. Moreover, it was proved in \cite[Remark 13.2]{bel} that every manifold tangentially homotopy equivalent to $\mathbb{C}P^2\times\mathbb{S}^3$ is diffeomorphic to $\mathbb{C}P^2\times \mathbb{S}^3$. In Theorem~\ref{tangetial_classification}, we also prove that every manifold tangentially homotopy equivalent to $\mathbb{C}P^2\times\mathbb{S}^6$ is oriented diffeomorphic to either $\mathbb{C}P^2\times\mathbb{S}^6$, $(\mathbb{C}P^2\times\mathbb{S}^6)\#\Sigma_{\alpha_1}$, or $(\mathbb{C}P^2\times\mathbb{S}^6)\#\Sigma_{\alpha_1}^{-1}$.

Here, as an immediate consequence of Theorem \ref{theo}, we obtain the following. 

\begin{corollary*}
 For $k=4$ and $5,$ the inertia group of $\mathbb{C}P^2\times\mathbb{S}^k$ equals $\Theta_{4+k}$. On the other hand, for $k=6$, the inertia group of $\mathbb{C}P^2\times\mathbb{S}^k$ equals $\mathbb{Z}/2.$
\end{corollary*}

	\subsection{Notation}
	\begin{itemize}
		\item Let $O_n$ be the orthogonal group, $PL_n$ is the simplicial group of piece-wise linear homeomorphisms of $\mathbb{R}^n$ fixing origin, $Top_n$ is the group of self homeomorphisms of $\mathbb{R}^n$ preserving origin, and $G_n$ denote the topological monoid of self-homotopy equivalences $\mathbb{S}^{n-1}\to \mathbb{S}^{n-1}.$ Denote by $O=\varinjlim~O_n$, $PL= \varinjlim PL_n$, $Top=\varinjlim Top_n,$ and $G=\varinjlim~ G_n$ \cite{kuiperlashof66,lashofrothenberg65,milgram}.
  
	  \item Let $G/O$ be the homotopy fiber of the canonical map $BO\to BG$ between the classifying spaces for stable vector bundles and stable spherical fibrations \cite[\S2 and \S3]{milgram}, and $G/PL$ be the homotopy fiber of the canonical map $BPL\to BG$ between the classifying spaces for stable $PL$-microbundles and stable spherical fibrations \cite{rudyak15}. Again, $Top/O$ is the homotopy fiber of the map $BO\to BTop$ between the classifying spaces for stable vector bundles and stable microbundles \cite[Theorem 10.1 Essay IV]{kirby}.
		\item  There are standard fiber sequences $\cdots\to\Omega G/Top\xrightarrow{w} Top/O\xrightarrow{\psi} G/O\xrightarrow{\phi} G/Top$ \cite{kirby} and $O\xrightarrow{\Omega \widetilde{J}}G\xrightarrow{j}G/O\xrightarrow{i}BO\xrightarrow{\widetilde{J}}BG$ \cite{gw} where $J=\widetilde{J}_*:\pi_n(O)\to\pi_n^s$ is the classical $J$-homomorphism.
		\item We write $X_{(p)}$ for localization of the space $X$ at prime $p.$ We use the notation $f_{(p)}$ when a map $f$ is localized at $p$ \cite{DPS05}. 
		
		\item $M(G)$ denotes the Moore spectrum associated to an abelian group $G$. 

            \item $f_M: M\to \mathbb{S}^n$ denotes the degree one map for an oriented $n$-manifold $M$.
           \item $\mathcal{E}(M)$ denotes the group of all homotopy classes self-homotopy equivalences of $M$.
           \item The notation $\{ -, - \}$ is used to denote the stable homotopy classes of maps between spectra.
	\end{itemize}
    
\subsection{Organization} In section \S \ref{aboutconcor}, we discuss the homotopy decomposition of $\mathcal{C}(M\times\mathbb{S}^k)$ and present some results related to $I_c(M\times\mathbb{S}^k)$. In section \S \ref{conlow}, we provide a stable decomposition of a closed, oriented, smooth $4$-manifold $M$, based on whether it admits a spin or non-spin structure. Using these decompositions, we compute both the concordance inertia group and the concordance smooth structure set of $M\times\mathbb{S}^k$ for $1\leq k\leq 10$. In section \S \ref{5dim}, we determine the stable homotopy type of a simply connected, closed, oriented, smooth $5$-manifold $M$ and extend our calculations to the product manifold $M\times\mathbb{S}^k$. Finally, in section \S \ref{application}, we apply our calculations to classify all manifolds that are homeomorphic to $\mathbb{C}P^2\times\mathbb{S}^k,$ where $4\leq k\leq 6$, up to diffeomorphism.
 
\subsection{Acknowledgements} 
The third author expresses gratitude to Douglas C. Ravenel and Diarmuid Crowley for their valuable help in clarifying doubts. The authors also thank the referee for insightful comments and suggestions that improved the clarity and presentation of the paper.

	
  \section{ Concordance inertia groups of the product manifolds \texorpdfstring{$M\times\mathbb{S}^k$}{}}\label{aboutconcor}
We start by recalling some facts about the Kervaire-Milnor group $\Theta_m$.
\begin{defn}\cite{milnor}
 \begin{itemize}
     \item[(a)] A homotopy $m$-sphere $\Sigma^m$ is an oriented smooth closed manifold homeomorphic to the standard unit sphere $\mathbb{S}^m$ in $\mathbb{R}^{m+1}.$
     \item[(b)] A homotopy $m$-sphere $\Sigma^m$ is said to be exotic if it is not diffeomorphic to $\mathbb{S}^m$.
     \item[(c)] Two homotopy $m$-spheres $\Sigma_{1}^m$ and $\Sigma_{2}^m$ are said to be equivalent if there exists an orientation preserving diffeomorphism $f :\Sigma_{1}^m\to \Sigma_{2}^m$.
 \end{itemize}
 The set of equivalence classes of homotopy $m$-spheres is denoted by  $\Theta_m$. Each $\Theta_m~(m\neq 4)$ is a finite abelian group under the connected sum operation.
 \end{defn}
 Assume $m\geq 5$ and consider the obvious forgetting information map $\mathcal{C}(\mathbb{S}^m)\to \Theta_m$, which is bijective. Observe that the abelian group structure on $\Theta_m$ given by the connected sum agrees with the one given by Kirby-Siebenmann \cite{kirby} via the identification 	$\mathcal{C}(\mathbb{S}^m)=\pi_m(Top/O)$. Note that for a closed, oriented \( m \)-manifold \( N^m \), there exists a degree one map \( f_N : N^m \to N^m/{(N^m-\text{int}(\mathbb{D}^m))}\simeq \mathbb{S}^m \), which is well-defined up to homotopy. 
 Composition with $f_N$ defines a homomorphism $f_N^*: \pi_m(Top/O)\to [N^m, Top/O].$ In terms of the Kirby-Siebenmann identifications $\Theta_m=\pi_m (Top/O)$ and $\mathcal{C}(N^m)=[N^m,Top/O]$, $f_N^*$ becomes $[\Sigma^m] \to [(N\#\Sigma^m, h_{\Sigma^m})]$, where $h_{\Sigma^m}: N^m\#\Sigma^m \to N^m$ is the canonical homeomorphism. Therefore, the \textit{concordance inertia group} $I_c(N^m)$ can be identified with the kernel $\mathit{ker}\left(\Theta_m\xrightarrow{(f_{N})^*}\mathcal{C}(N^m)\right).$  Specifically, for $N= M\times\mathbb{S}^k$, we have the following:
	\begin{lemma}\label{rmk2.5}
		Let $M$ be a closed, oriented, smooth $n$-manifold and $\Sigma^k f_{M}: \Sigma^k M\to \s^{n+k}$ be the $k$-fold suspension of the degree one map $f_{M}:M\to \s^n.$ Then the concordance inertia group 
		$I_c\left(M\times\s^k\right)$ is equal to $\mathit{ker}\left(\Theta_{n+k}\xrightarrow{(\Sigma^k f_{M})^*}[\Sigma^k M,Top/O]\right)$ for $n + k\geq 5.$
	\end{lemma}
	\begin{proof}
		We observe that there is a split short exact sequence along an $H$-group $Y$
		\begin{equation}\label{note2.3}
			0\to[\Sigma^k M,Y]\xrightarrow{p^*}[M\times\s^k,Y]\xrightarrow{i^*}[M\vee\s^k,Y]\to0
		\end{equation}
		which is induced from the cofiber sequence
		$$M\vee\s^k\xhookrightarrow{i}M\times\s^k\xrightarrow{p} M\times\s^k/M\vee\s^k\simeq\Sigma^k M.$$ 
Note that the map $p$ fits into the following homotopy commutative diagram:
\[
\begin{tikzcd}
	{M\times\mathbb{S}^k} & {\Sigma^k M} \\
	{\mathbb{S}^{n+k}.}
	\arrow["p", from=1-1, to=1-2]
	\arrow["{f_{M\times\mathbb{S}^k}}"', from=1-1, to=2-1]
	\arrow["{\Sigma^k f_{M}}", from=1-2, to=2-1]
\end{tikzcd}\]
This, in turn, induces the following commutative diagram on homotopy classes of maps:
\[
\begin{tikzcd}
	{[M\times\mathbb{S}^k, Top/O]} && {[\Sigma^k M, Top/O]} \\
	\\
	{[\mathbb{S}^{n+k}, Top/O].}
	\arrow["{p^*}"', from=1-3, to=1-1]
	\arrow["{f_{M\times\mathbb{S}^k}^*}", from=3-1, to=1-1]
	\arrow["{(\Sigma^k f_{M})^*}"', from=3-1, to=1-3]
\end{tikzcd}\]
Since $p^*$ is injective by \eqref{note2.3}, the conclusion follows.
				
\end{proof}
By combining Kirby and Siebenmann result and the split short exact sequence (\ref{note2.3}), we obtain the following corollary.
   \begin{cor}\label{concorMtimesSk}
       For any closed, oriented, smooth manifold $M$,
       $$\mathcal{C}(M\times\mathbb{S}^k)\cong [M,Top/O]\oplus\pi_k(Top/O)\oplus[\Sigma^k M,Top/O],$$ where $\mathit{dim}(M)+k\geq 5.$
   \end{cor}
 
	Now the next lemma provides a pathway to interpret the concordance inertia group of $M\times\mathbb{S}^k,$ if the stable homotopy type of $M$ is known.

        \begin{lemma}\label{concordanceinertia}
Let $M$ be an $n$-dimensional closed, oriented, and smooth manifold. Suppose $M$ is stably homotopy equivalent to $X\vee Z$, where $X$ is the cofibre of a certain map $f: \mathbb{S}^{n-1} \to \bigvee\limits_{i=1}^w A_i$ with $w\geq 2$, and the dimensions of both  $\bigvee\limits_{i=1}^w A_i$ and $Z$ are less than $n$. Then, for any integer $k\geq 1$ such that $n+k\geq 5,$ the concordance inertia group $I_c(M\times \mathbb{S}^k)$ is isomorphic to $$\mathit{Im}\left(\bigoplus\limits_{i=1}^w [\Sigma^{k+1} A_i,Top/O]\xrightarrow{\bigoplus\limits_{i=1}^w (\Sigma^{k+1} f_i)^*}\pi_{n+k}(Top/O)\right),$$ where $f_i = Pr_i \circ f, 1\leq i\leq w$ are the compositions of $f$ with the projection maps $Pr_i$ from $\bigvee\limits_{i=1}^w A_i $ onto its $i^{th}$ factor.      
        \end{lemma}
        \begin{proof}
        By assumption, there exists a positive integer $j$ such that for any $l\geq j,$ we have a homotopy equivalence $\phi_{l}:\Sigma^l M \to \Sigma^l X \vee \Sigma^l Z.$ Since $Top/O$ is an infinite loop space, there exists a topological space $Y$ such that $\Omega^j Y=Top/O$. Now by Lemma \ref{rmk2.5}, we have that $$I_c(M\times\mathbb{S}^k)\cong \mathit{ker}\left(\pi_{n+k+j}(Y)\xrightarrow{(\Sigma^{k+j} f_{M})^*} [\Sigma^{k+j} M,Y]\right).$$
        Observe that the $(k+j)$-fold suspension of the degree one map $\Sigma^{k+j} f_M : \Sigma^{k+j} M\to \mathbb{S}^{n+k+j}$ is homotopic to the composition $\lambda \circ \Sigma^{k+j}f_X \circ Pr\circ \phi_{k+j}$, where $Pr: \Sigma^{k+j} X \vee \Sigma^{k+j} Z\to \Sigma^{k+j}X$ is the the projection map, $\Sigma^{k+j}f_X: \Sigma^{k+j}X \to \Sigma^{k+j}\mathbb{S}^{n}$ is the $(k+j)$-fold suspension of the pinch map onto the top cell $f_{X}:X\to \mathbb{S}^{n}$, and $\lambda:\mathbb{S}^{n+k+j}\to \mathbb{S}^{n+k+j}$ is either the identity or the reflection map. This implies that $$I_c(M\times\mathbb{S}^k)\cong \mathit{ker}\left(\pi_{n+k+j}(Y)\xrightarrow{(\Sigma^{k+j} f_X)^*} [\Sigma^{k+j}X,Y]\right).$$ This can be equivalently expressed as the image of the homomorphism $\bigoplus\limits_{i=1}^w (\Sigma^{k+j+1} f_i)^* :\bigoplus\limits_{i=1}^w [\Sigma^{k+j+1} A_i,Y]\to\pi_{n+k+j}(Y),$ by using the long exact sequence induced from the cofiber sequence $\Sigma^{k+j} \mathbb{S}^{n-1} \xrightarrow{\Sigma^{k+j} f} \Sigma^{k+j}(\bigvee\limits_{i=1}^w A_i)\hookrightarrow \Sigma^{k+j} X$ along $Y$. Consequently, the latter group is isomorphic to the image of the map $\bigoplus\limits_{i=1}^w (\Sigma^{k+1}f_i)^* :\bigoplus\limits_{i=1}^w [\Sigma^{k+1} A_i,Top/O]\to \pi_{n+k}(Top/O)$. This completes the proof.
        \end{proof}

        We notice from the above lemma that the calculation of the concordance inertia group of $M\times \mathbb{S}^k$ depends on the stable homotopy type of $M.$ Further, if $M$ has a stable homotopy decomposition $M \simeq X\vee Y$ as described in Lemma \ref{concordanceinertia}, then it suffices to determine the image of each component in the suspension of the top cell attaching map of $X$ at the $Top/O$ level.
	 In order to make explicit computations, we require the knowledge of a few stable homotopy groups and their mapping cones. 
	\[\pi_0^s = \Z,~ \pi_1^s = \Z/2\{\eta \},~\pi_2^s=\Z/2\{\eta^2\},~\pi_3^s=\Z/{24}\{\tau\},~ \pi_4^s=0,~\pi_5^s =0,\]
 where $(\tau)_{(2)}=\nu,$ and $(\tau)_{(3)}=\alpha_1.$
	Let $b\in \Z$ be a non-zero integer. The mapping cone of $b: \s^0 \to \s^0$ in spectra is denoted by the Moore spectrum $M(\Z/b)$. Its homotopy groups in the above degrees are described as  
	\[\pi_0(M(\Z/b)) = \Z/b,~ \pi_1(M(\Z/b)) = \begin{cases}\Z/2\{i \circ\eta \} &\mbox{if $b$ is even}\\ 0 &\mbox{ otherwise,}\end{cases}\]
	\[\pi_2(M(\Z/b))=\begin{cases}\Z/4\{\tilde{\eta}_b\} &\mbox{if $2\mid b$ but $4\nmid b$}\\ \Z/2\{\tilde{\eta}_b\}\oplus \Z/2\{i\circ \eta^2\} &\mbox{if } 4\mid b \\ 0 &\mbox{otherwise,} \end{cases} \] 
		\[\pi_3(M(\Z/b))=\begin{cases} \Z/\gcd(24,b)\{i\circ \tau\}\oplus \Z/2\{\tilde{\eta_b^2}\} &\mbox{if } 2 \mid b \\ \Z/\gcd(24,b)\{i\circ \tau\} &\mbox{if } 2\nmid b,\end{cases}\] 
	\[ ~ \pi_4(M(\Z/b))=\Z/\gcd(24,b)\{\tilde{\tau}\},~\pi_5(M(\Z/b)) =0.\]
	In the above $\tilde{\eta}_{b}$ (respectively $\tilde{\tau}$) refers to a homotopy class which gives $\eta$ (respectively $\tau$) after composition onto $\s^1$. We recall that the Hopf map $\eta$ is detected by $Sq^2$, and that the mapping cone $Cone(\eta)\simeq \Sigma^{-2}\C P^2$. The square of $\eta$, $\eta^2$ is detected by the secondary cohomology operation \cite{harper, tangora}
	\[ \Theta : \mathit{ker}(Sq^2)\cap \mathit{ker}(Sq^1) \Big(\subset  H^i(X;\Z/2)\Big) \to H^{i+3}(X;\Z/2)/{Sq^3 H^i(X;\mathbb{Z}/2)+ Sq^2 H^{i+1}(X;\mathbb{Z}/2)}, \]
	associated to the Adem relation $Sq^2 Sq^2 = Sq^3 Sq^1$ in the Steenrod algebra. 

Additionally, we need to understand the homotopy class of the maps from the Moore spectrum $ M(\mathbb{Z}/{2^t})$ to $ M(\mathbb{Z}/{2^r})$, for some $t, r\in \mathbb{N}$ with $t\leq r$. As noted in \cite[Page 260-261]{SO84}, there exists a canonically defined map $\lambda_{t,r}: M(\mathbb{Z}/{2^t})\to M(\mathbb{Z}/{2^r})$ such that following diagram commutes:
\begin{equation}\label{diagrammoore}
\begin{tikzcd}
{\s^0}\arrow[r,"\times 2^r"]\arrow[d,-,double equal sign distance,double]&{\s^0}\arrow[hookrightarrow]{r}{i^{2^r}_0}&{M(\z/{2^r})}\arrow[r,"q^{2^r}_1"]&{\Sigma\s^0}\arrow[r,"\times 2^r"]\arrow[d,-,double equal sign distance,double]&{\Sigma\s^0}\\
{\s^0}\arrow[r,"\times 2^t"']&{\s^0} \arrow[hookrightarrow]{r}[swap]{i^{2^t}_0}\arrow[u,"\times 2^{r-t}"']&{ M(\z/{2^t})}\arrow[r,"q^{2^t}_1"']\arrow[u,dashed,"\lambda_{t,r}"']&{\Sigma\s^0}\arrow[r,"\times 2^t"']&{\Sigma\s^0.}\arrow[u,"\times 2^{r-t}"']
\end{tikzcd}
\end{equation}
Furthermore, it follows from \cite[Page 261]{SO84} that the induced homomorphism $(\lambda_{t,r})_*: \pi_2(M(\mathbb{Z}/{2^t}))\to \pi_2(M(\mathbb{Z}/{2^r}))$ satisfies 
\begin{equation}\label{Bchirelation}
   (\lambda_{t,r})_*(\tilde{\eta}_{2^t})=\lambda_{t,r}\circ \tilde{\eta}_{2^t}=\tilde{\eta}_{2^r}.
\end{equation}
In addition, \cite[Lemma 7]{SO84} shows that 
\begin{equation*}
    \{M(\mathbb{Z}/{2^t}), M(\mathbb{Z}/{2^r})\}=\begin{cases}
        \mathbb{Z}/4\{\lambda_{1,1}\} &\mbox{if } r=t=1,\\
        \mathbb{Z}/2^{min(r,t)}\{\lambda_{t,r}\}\oplus \mathbb{Z}/2\{i^{2^r}_0\circ \eta \circ q^{2^t}_{1}\} &\mbox{otherwise.}
    \end{cases}
\end{equation*}
\section{Smooth structures on \texorpdfstring{$ M^4 \times \mathbb{S}^k$}{}} \label{conlow}

Let $M$ be a closed, smooth, and oriented 4-manifold. In this section, we first determine the stable homotopy type of $M$. Following this, we compute $I_c(M\times \mathbb{S}^k)$ and $\mathcal{C}(M\times \mathbb{S}^k)$. By applying Poincaré duality and the Universal Coefficient Theorem, we obtain

	\begin{equation}\label{hom} 
		H_l(M;\z) =
		\begin{cases}
			\oplus\mathbb{Z}^{m}\oplus \bigoplus\limits_{j=1}^n\z/{b_j} & l=1\\
			\oplus \mathbb{Z}^{d}\oplus\bigoplus\limits_{j=1}^n\z/{b_j} & l=2\\
			\oplus\mathbb{Z}^{m} & l=3\\
			\mathbb{Z} & l=0,4\\
			0 & \text{otherwise}\\
		\end{cases}       
	\end{equation}
	where $m, d\geq0$ are any integers and each $b_j$ is a prime power. We make certain choices in the above decomposition. Suppose that $u_1,\cdots, u_m$ be a basis of the free part of $H_1(M;\z)$, and $v_1,\cdots, v_d$ be a basis of the free part of $H_2(M;\z)$. Choose a basis $\mu_1,\cdots, \mu_m$ of $H^1(M;\Z)\cong \Z^m$ such that $\langle \mu_i, u_j\rangle=\delta_{ij}$. The basis of $H_3(M;\z)$ is then fixed as $[M]\cap \mu_1,\cdots, [M]\cap \mu_m$. Let $a_i$ denote a generator of the $\Z/b_i$ factor of $H_1(M;\z)$. These come equipped with maps $ \hat{q}_i : M \to K(\mathbb{Z}/b_i, 1) $ such that $ a_i $ maps to $ 1 $ in $ H_1(K(\mathbb{Z}/b_i, 1); \mathbb{Z}) $, while $ a_j $ for $ j \neq i $, and all $ u_j $, map to $ 0 $. 
    Define $\alpha_i \in H^2(M)$ as the image 
	\[1 \in \Z/b_i \cong H^2(K(\Z/b_i,1);\z) \stackrel{\hat{q_i}^\ast}{\to} H^2(M;\z).\] 
	Define $k_i = [M]\cap \alpha_i$. With $\Z/2$-coefficients, define $\nu_i$ to be a dual basis of $v_i$ which evaluates to $0$ on the torsion classes. If $2\mid b_i$, let $c_i\in H_2(M;\Z/2)$ denote a choice of a class such that $\hat{q_i}_\ast(c_i)$ is a generator of $H_2(K(\Z/b_i,1);\Z/2)$. In this case, also define $\chi_i\in H^2(M;\Z/2)$ by the equation 
	\[ \langle \chi_i, k_j\rangle = \delta_{ij}, ~ \langle \chi_i,v_j\rangle = 0, ~ \langle \chi_i, c_j\rangle = 0.\] 
	We thus have for some $r\leq n$, 
	\[ H^2(M;\Z/2) \cong \Z/2\{\nu_1,\cdots, \nu_d, \bar{\alpha_1},\cdots,\bar{\alpha_r},\chi_1,\cdots,\chi_r\},\]
 where $\bar{\alpha_i}$ is the $\pmod{2}$ reduction of the $\alpha_i$. In terms of the cup product $H^2(M;\Z/2)\otimes H^2(M;\Z/2) \to H^4(M;\Z/2) \cong \Z/2$, we have 
	\[\alpha_i \nu_j=0,~ \chi_i\alpha_j=\delta_{ij}.\]
	
	\subsection{The stable homotopy type of an orientable \texorpdfstring{$4$}{}-manifold} \label{stabletypeM^4}
	We use a minimal cell structure on the stable homotopy type of $M$ \cite[\S 4.C]{hatcher}. The $2$-skeleton is then clearly homotopy equivalent to  
	$$M^{(2)}\simeq \bigvee_{i=1}^m \s^1\vee \bigvee_{j=1}^n \left(\Sigma M(\Z/b_j)\vee \s^2\right)\vee \bigvee_{l=1}^d \s^2.$$
	We now consider the attachment of the $3$-cells which are stably a sum of maps onto the wedge summands. We further note that $Sq^2: H^1(M;\z/2)\to H^3(M;\z/2)$ equals to zero. This means that the attaching map does not involve any factor of $\eta$. It follows that the $3$-skeleton has the following stable homotopy type. 
	\begin{equation*}
		M^{(3)}\simeq \bigvee_{i=1}^m \left(\s^1\vee\s^3\right)\vee\bigvee_{j=1}^n(\Sigma M(\z/{b_j})\vee \Sigma^2 M(\z/{b_j}))\vee\bigvee_{l=1}^d\s^2.
	\end{equation*}
	Finally, the stable homotopy type of $M$ is determined by the attaching map $\phi_3: \s^3 \to M^{(3)}$ for the $4$-cell of $M$. The formula for the homology implies that the attaching map becomes trivial after quotienting out the $2$-skeleton. Thus, we consider the factorization of $\phi_3$ as 
	\[\s^3 \stackrel{\phi}{\to} M^{(2)} \to M^{(3)}.\]
	The possibilities for the map $\phi$ are computed via two observations. Applying \cite[Corollary 1.8]{GM68} for $a=b=2$ on a $1$-dimensional class, we have $\Theta : H^1(M;\Z/2) \to H^4(M;\Z/2)$ is $0$. This implies that $\phi$ maps trivially onto the factors given by $\s^1$ and the $0$-cell of $\Sigma M(\Z/b_j)$. Therefore, the non-trivial factors of the map $\phi$ may only be \\
	1) As a non-zero multiple of $\eta$ onto some of the $d$-copies of $\s^2$. \\
	2) As a non-zero multiple of $i\circ\eta$ onto one of the copies of $\Sigma^2 M(\Z/b_j)$. \\
	3) As a non-zero multiple of $\tilde{\eta}_{b_j}$ onto one of the copies of $\Sigma M(\Z/b_j)$. \\
	If any of the cases above occur except when $2 \mid b_j$ but $4\nmid b_j$ and the multiple is $2 \tilde{\eta}_{b_j}$, the operation $Sq^2 : H^2(M;\Z/2) \to H^4(M;\Z/2)$ is non-trivial, which means by Wu's formula that the manifold is not spin. However, this exception cannot occur as $2\tilde{\eta} = i \circ \eta^2$. We also observe that if $2\mid b_j$ but $4\nmid b_j$, the last case cannot arise for a $4$-manifold, as then the composition 
	\[ H^1(M;\Z/2) \stackrel{Sq^1}{\to} H^2(M;\Z/2) \stackrel{Sq^2}{\to} H^4(M;\Z/2)\]
	would be non-trivial. However, this composition maps a class $\alpha \mapsto \alpha^4$, and we then have $Sq^1(\alpha^3)=\alpha^4$, which contradicts $w_1(M)=0$ by Wu's formula. In fact, we will now observe that 3) cannot occur at all and if 2) occurs, 1) will also occur. The latter case is clear as the cohomology classes $\chi_j$ which evaluate non-trivially on the factors corresponding to the $\Sigma^2 M(\Z/b_j)$ factors for $2\mid b_j$ are not integral, and the fact that $w_2(M)$ is the $\pmod{2}$ reduction of an integral class \cite[Page 170]{Sco} and \cite{TeVo}. More explicitly, let $\tilde{w} \in H^2(M;\Z)$ be such that $w_2(M)=\tilde{w} \pmod{2}$. Then, this fits into the following commutative diagram if 1) does not occur. 
    \begin{figure}[htbp]
\centering
\begin{tikzpicture}[
    node distance=1.5cm and 2cm,
    every node/.style={inner sep=1pt},
    mapsto/.style={-{Straight Barb[angle=60:3pt 3]}, dashed},
    hook/.style={right hook-latex},
    thick
]

\node (ExtZ) {$\text{Ext}(H_1(M),\mathbb{Z})$};
\node[right=2.5cm of ExtZ] (H2Z) {$H^2(M;\mathbb{Z})$};
\node[right=2cm of H2Z] (HomZ) {$\text{Hom}(H_2(M),\mathbb{Z})$};

\node[below=1.8cm of ExtZ] (ExtZ2) {$\text{Ext}(H_1(M),\mathbb{Z}/2)$};
\node[below=1.8cm of H2Z] (H2Z2) {$H^2(M;\mathbb{Z}/2)$};
\node[below=1.8cm of HomZ] (HomZ2) {$\text{Hom}(H_2(M),\mathbb{Z}/2),$};

\node[below right=0.05cm and .1cm of H2Z] (wtilde) {$\tilde{w}$};
\node[right=1.9cm of wtilde] (zero1) {$0$};
\node[below=1cm of wtilde] (w2) {$w_2(M)$};
\node[right=1.5cm of w2] (zero2) {$0$};

\draw[->] (ExtZ) -- (H2Z);
\draw[->] (H2Z) -- (HomZ);
\draw[->] (ExtZ) -- (ExtZ2);
\draw[->] (H2Z) -- (H2Z2);
\draw[->] (HomZ) -- (HomZ2);
\draw[->] (ExtZ2) -- (H2Z2);
\draw[->] (H2Z2) -- (HomZ2);

\draw[|->] (wtilde) -- (zero1);
\draw[|->] (wtilde) -- (w2);
\draw[|->] (w2) -- (zero2);
\draw[|->] (zero1) -- (zero2);

\end{tikzpicture}
\end{figure}

which implies that the image of $w_2(M)$ in $Hom(H_2(M),\Z/2)$ is $0$.
	
	Suppose that the composite 
	\[ \s^3 \to M^{(3)} \stackrel{q_j}{\to} \Sigma M(\Z/b_j)\]
	equals $\tilde{\eta}$  for some $j$ such that $2\mid b_j$. Viewing unstably, the map $q_j$ yields  $\hat{q_j}:M^{(3)} \to K(\Z/b_j,1)$ which classifies $1\in \Hom(\Z/b_j,\Z/b_j)\subset \Hom(H_1(M),\Z/b_j)$.  As $Sq^2(\alpha)=\alpha^2$ for a degree $2$ class $\alpha$, the assumption implies that the class $\alpha_j \in H^2(M;\Z/2)$ corresponding to the generator of $\Ext(\Z/b_j,\Z/2)\Big(\subset \Ext(H_1(M),\Z/2) \subset H^2(M;\Z/2)\Big)$ satisfies $\alpha_j^2$ is the generator of $H^4(M;\Z/2)\cong \Z/2$. The entire information fits into the following diagram:
	\[\xymatrix{ H^2(K(\Z/b_j,1);\Z/2) \ar[d]^{Sq^2} \ar[r]^-{\hat{q}^\ast}& H^2(M;\Z/2) \ar[d]^{Sq^2} & H^2(M;\Z) \ar[l] \ar[d]^{(~)^2}\\
		H^4(K(\Z/b_j,1); \Z/2) \ar[r]^-{\hat{q}^\ast}  & H^4(M;\Z/2) & {H^4(M;\Z).} \ar[l]}\]
	As $\Ext(\Z/b_j,\Z) \to \Ext(\Z/b_j, \Z/2)$ is surjective, the class $\alpha_j$ lifts to a $b_j$-torsion class in $H^2(M;\Z)$. Therefore, its square is $0$, and thus $Sq^2(\alpha_j)=0$. This rules out the option 3) for an orientable non-spin manifold $M$. 
	
	We conclude that for a non-spin $4$-manifold $M$,  $\phi$ may map by $\eta$ to more than one copy of $\s^2$, as well as the maps of the type (2). However, we may change the wedge sum decomposition of $M^{(3)}$, so that $\phi$ maps to exactly one copy of $\s^2.$ This is done by considering the homotopy class $\s^2 \to M^{(3)}$ as the sum of the non-trivial $\s^2$-factors and the inclusions coming from the $M(\Z/b_j)$-factors of type (2). This means that after such a choice $\phi$ maps by $\eta$ to exactly one copy of $\s^2$.
	
	Therefore, we obtain exactly two cases  (noting that $Cone(\eta : \s^3 \to \s^2) \simeq \C P^2$).
	\begin{enumerate}
		\item[(a)] If the manifold $M$ is spin, then the stable homotopy type of $M$ is 
        \begin{equation}\label{spinM^4}
         M\simeq \s^4\vee  \bigvee_{i=1}^m \left(\s^1\vee\s^3\right)\vee\bigvee_{j=1}^n(\Sigma M(\z/{b_j})\vee \Sigma^2 M(\z/{b_j}))\vee\bigvee_{l=1}^d\s^2 .
        \end{equation}
		\item[(b)] If the manifold $M$ is non-spin, then its stable homotopy type is 
       \begin{equation}\label{nonspinM^4}
           M\simeq \C P^2\vee  \bigvee_{i=1}^m \left(\s^1\vee\s^3\right)\vee\bigvee_{j=1}^n(\Sigma M(\z/{b_j})\vee \Sigma^2 M(\z/{b_j}))\vee\bigvee_{l=1}^{d-1}\s^2 .
       \end{equation}
	
	\end{enumerate}
	
	\subsection{Concordance classes for \texorpdfstring{$M\times\s^k$}{}}
	
	From the above section, we have $M\simeq M^{(3)}\cup_{\phi_3}\mathbb{D}^4$ as an object in the stable homotopy category,  where $\phi_3$ is null homotopic if $M$ is a spin manifold and $\phi_3= i_d\circ\eta$ if $M$ is a non-spin manifold. Here $i_d$ refers to the map which is the inclusion of the $d^{th}$ $\s^2$ factor. This, together with Lemma \ref{concordanceinertia}, allows us to compute $I_c(M\times\mathbb{S}^k)$ by analyzing only the image of $\eta^*:\pi_{k}(Top/O)\to \pi_{k+1}(Top/O).$ In particular, we prove the following lemma.
	
 We first recall the relevant homotopy groups of the infinite loop spaces $Top/O, G/O$, and $G/Top$, which will be used throughout this article.\\
The homotopy groups of $G/Top$ are given by \cite[Page 422]{Luck_surgery}:
\[\pi_n(G/Top)\cong \begin{cases}
    \mathbb{Z}&\mbox{if }  n \equiv 0 \mod 4, \\
    0 &\mbox{if } n\equiv 1\mod 4,\\
    \mathbb{Z}/2 &\mbox{if } n\equiv 2 \mod 4,\\
    0 &\mbox{if } n\equiv 3\mod 4.
\end{cases}\]
The low-dimensional homotopy groups of $G/O$, as compiled from \cite{Luck_surgery, Sullivan96, Mimura63}, are listed in the following table.

\begin{table}[h]
\centering
\resizebox{\textwidth}{!}{%
\begin{tabular}{|c|*{10}{c}|}
\hline
$k$ & 1 & 2 & 3 & 4 & 5 & 6 & 7 & 8 & 9 & 10 \\
\hline
$\pi_k(G/O)$ & $0$ & $\mathbb{Z}/2$ & $0$ & $\mathbb{Z}$ & $0$ & $\mathbb{Z}/2$ & $0$ & $\mathbb{Z}\oplus\mathbb{Z}/2$ & $2\mathbb{Z}/2$ & $\mathbb{Z}/6$ \\
\hline
\hline
$k$ & 11 & 12 & 13 & 14 & 15 & 16 & 17 & 18 & 19 & 20 \\
\hline
$\pi_k(G/O)$ & $0$ & $\mathbb{Z}$ & $\mathbb{Z}/3$ & $2\mathbb{Z}/2$ & $\mathbb{Z}/2$ & $\mathbb{Z}\oplus\mathbb{Z}/2$ & $3\mathbb{Z}/2$ & $\mathbb{Z}/2\oplus\mathbb{Z}/8$ & $\mathbb{Z}/2$ & $\mathbb{Z}\oplus\mathbb{Z}/24$ \\
\hline
\end{tabular}%
}
\end{table}

We also include the homotopy groups of $Top/O$ as given in \cite{milnor}.

\begin{table}[h]
\centering
\resizebox{\textwidth}{!}{%
\begin{tabular}{|c|*{10}{c}|}
\hline
$k$ & 1 & 2 & 3 & 4 & 5 & 6 & 7 & 8 & 9 & 10 \\
\hline
$\pi_k(Top/O)$ & $0$ & $0$ & $\mathbb{Z}/2$ & $0$ & $0$ & $0$ & $\mathbb{Z}/28$ & $\mathbb{Z}/2$ & $3\mathbb{Z}/2$ & $\mathbb{Z}/6$ \\
\hline
\hline
$k$ & 11 & 12 & 13 & 14 & 15 & 16 & 17 & 18 & 19 & 20 \\
\hline
$\pi_k(Top/O)$ & $\mathbb{Z}/992$ & $0$ & $\mathbb{Z}/3$ & $\mathbb{Z}/2$ & $\mathbb{Z}/2\oplus\mathbb{Z}/8128$ & $\mathbb{Z}/2$ & $4\mathbb{Z}/2$ & $\mathbb{Z}/2\oplus\mathbb{Z}/8$ & $\mathbb{Z}/130816\oplus\mathbb{Z}/2$ & $\mathbb{Z}/24$ \\
\hline
\end{tabular}%
}
\end{table}

	\begin{lemma}\label{lemma2.17}
		The image of $\eta^*:\pi_{k}(Top/O)\to\pi_{k+1}(Top/O)$ is given in the following table:
		\begin{center}
			\begin{tabular}{|c |c|c|c|c| c|c|c|c|c|c|c|c|c|c|c|c|} 
				\hline
				{$k=$}&1 \text{to} 6 &7 & 8 & 9 & 10 & 11 & 12 & 13 & 14 &15&16&17&18&19\\
				\hline
				\multirow{1}{4em}{$\mathit{Im}(\eta^*)$}&0 &0 & $\z/2$ & $\mathbb{Z}/2$& {$\z/2$} & {$0$} & {$0$} & {$0$} & {$\z/2$} & {$0$} & $\z/2$ & $\z/2\oplus\z/2$ & $\mathbb{Z}/2$ & $0$\\
				\hline
			\end{tabular}
		\end{center}
	\end{lemma}
	\begin{proof}
Since \(\pi_k(Top/O)\) vanishes for all \(1\leq k\leq 6\) with \(k \neq 3\), and also for \(k=12\), the result follows for \(k=1,2,3,4,5,6,11,\) and \(12\).

As \(\pi_{13}(Top/O)_{(2)}\) is zero, the map \(\eta_{13}^* : \pi_{13}(Top/O) \to \pi_{14}(Top/O)\) is trivial.

Consider the following commutative diagram:
\begin{equation}\label{commutative_diagram}
  \begin{tikzcd}
	{\pi_k(Top/O)} && {\pi_{k+1}(Top/O)} \\
	\\
	{\pi_k(G/O)} && {\pi_{k+1}(G/O)} \\
	\\
	{\pi_k^s} && {\pi_{k+1}^s.}
	\arrow["{\eta_k^*}", from=1-1, to=1-3]
	\arrow["{\psi_*}"', from=1-1, to=3-1]
	\arrow["{\psi_*}", from=1-3, to=3-3]
	\arrow["{\eta_k^*}"', from=3-1, to=3-3]
	\arrow["{j_*}", from=5-1, to=3-1]
	\arrow["{\eta_k^*}"', from=5-1, to=5-3]
	\arrow["{j_*}"', from=5-3, to=3-3]
\end{tikzcd}
\end{equation}
Since $\pi_7(G/O)=0$ and $\psi_*: \pi_8(Top/O)\to\pi_8(G/O)$ is injective, it follows from the commutative diagram above that the map $\eta_7^*:\pi_7(Top/O)\to\pi_8(Top/O)$ is trivial.

For $k=8$, we note from \cite[Page 189]{toda} that the image of $\eta_8^*:\pi_8^s\to \pi_9^s$ is $\mathbb{Z}/2\{\nu^3\}\oplus \mathbb{Z}/2\{\eta\epsilon\}$. Since $j_*(\Bar{\nu})=j_*(\epsilon)$ in $\pi_8(G/O)$ \cite{ravenel}, we obtain from the lower rectangle of the diagram \eqref{commutative_diagram} that the image of the map $\eta_8^*: \pi_8(G/O)\to\pi_9(G/O)$ restricted to the torsion part of $\pi_8(G/O)$ is $\mathbb{Z}/2$. Now, the result follows from the top rectangle \eqref{commutative_diagram} as $j_*(\pi_8^s)=\psi_*(\pi_8(Top/O))$.


For \( k = 9 \), the element \( \mu \in \pi_9^s \) maps to \( \eta\mu \in \pi_{10}^s \) under the map \( \eta_9^* \). The image of \( \mu \) under \( j_* : \pi_9^s \to \pi_9(G/O) \) is nonzero \cite[Theorem 1.1.14]{ravenel}, and the map \( \psi_* : \pi_9(Top/O) \to \pi_9(G/O) \) is onto. Furthermore, both \( \psi_* : \pi_{10}(Top/O) \to \pi_{10}(G/O) \) and \( j_* : \pi_{10}^s \to \pi_{10}(G/O) \) are isomorphisms. It then follows from the commutativity of diagram~\eqref{commutative_diagram} that the image of \( \eta_9^* : \pi_9(Top/O) \to \pi_{10}(Top/O) \) is isomorphic to \( \mathbb{Z}/2 \), generated by the element corresponding to \( \eta\mu \in \pi_{10}^s \).

We now turn to the case $k=10$. Since $Top/PL\simeq K(\mathbb{Z}/2,3)$ \cite[Page 198]{kirby}, the natural map $PL/O\to Top/O$ induces isomorphism on the homotopy groups in degree $i\geq 4$. Hence to compute the image of the map $\eta_{10}^*: \pi_{10}(Top/O)\to \pi_{11}(Top/O)$, it suffices to determine the image of $\eta_{10}^*: \pi_{10}(PL/O)\to \pi_{11}(PL/O)$. Since $\pi_{10}(PL/O)=\mathbb{Z}/6$ and $\pi_{11}(PL/O)=\mathbb{Z}/{992}$, the image of $\eta_{10}^*: \pi_{10}(PL/O)\to \pi_{11}(PL/O)$ is atmost $\mathbb{Z}/2$. We claim that the image of the map $\eta_{10}^*: \pi_{10}(PL/O)\to \pi_{11}(PL/O)$ is exactly $\mathbb{Z}/2$. To prove this, we consider the following commutative diagram  with exact rows induced from the fiber sequence $PL/O\to BO\to BPL$
\[
\begin{tikzcd}
    &{0}\arrow[r]\arrow[d]& {\pi_{10}(PL)}\arrow[d,"\eta_{10}^*"']\arrow[r,"\cong"]&{\pi_{10}(PL/O)}\arrow[r,"0"]\arrow[d,"\eta_{10}^*"]& {\pi_{10}(BO)}\arrow[d]\\
    {0}\arrow[r]&{\pi_{12}(BO)}\arrow[r]&{\pi_{11}(PL)}\arrow[r,"\beta"']&{\pi_{11}(PL/O)}\arrow[r]&{0,}
\end{tikzcd}\]
where $\pi_{10}(PL)\cong \pi_{10}(PL/O)$ \cite{Hirch,hirsch} and $\pi_{11}(PL)\cong \mathbb{Z}\oplus\mathbb{Z}/8$ \cite[Theorem 1.4]{gw2}. Since the restriction of \( \beta : \pi_{11}(PL) \to \pi_{11}(PL/O) \) to the torsion subgroup of $\pi_{11}(PL)$ is an inclusion \cite[Page 307]{gw2}, it follows from the diagram that the image of \( \eta_{10}^* : \pi_{10}(PL/O) \to \pi_{11}(PL/O) \) is \( \mathbb{Z}/2 \) if and only if the image of \( \eta_{10}^* : \pi_{10}(PL) \to \pi_{11}(PL) \) is \( \mathbb{Z}/2 \). The map $\eta_{10}^*:\pi_{10}(PL)\to \pi_{11}(PL)$ fits into the following commutative diagram:
\[\begin{tikzcd}
    &{0}\arrow[r]&{\pi_{10}(PL)}\arrow[r,"\cong"]\arrow[d,"\eta_{10}^*"']&{\pi_{10}^s}\arrow[r]\arrow[d,"\eta_{10}^*"]&{\pi_{10}(G/PL)}\\
    {0}\arrow[r]&{\pi_{12}(G/PL)}\arrow[r]&{\pi_{11}(PL)}\arrow[r]&{\pi_{11}^s}\arrow[r]&{0,}
\end{tikzcd}\]
where the rows are induced from the fibration $G/PL\to BPL\to BG$; the isomorphism $\pi_{10}(PL)\cong \pi_{10}^s$ holds since $\pi_{10}(PL)\cong \pi_{10}(PL/O)$. 
Since the image of $\eta_{10}^*: \pi_{10}^s\to \pi_{11}^s$ is $\mathbb{Z}/2$ \cite{toda}, the above diagram shows that the map $\eta_{10}^*:\pi_{10}(PL)\to \pi_{11}(PL)$ has image $\mathbb{Z}/2$. This completes the proof for $k=10$.

For \( k = 14 \), it follows from \cite[Page 189-190]{toda} that the generators \( \kappa \) and \( \sigma^2 \) of \( \pi_{14}^s \) are mapped under \( \eta_{14}^* \) to \( \eta\kappa \) and \( 0 \) in $\pi_{15}^s$, respectively. Moreover, the map \( j_* : \pi_{14}^s \to \pi_{14}(G/O) \) is an isomorphism, since \( \pi_{15}(BO) = 0 \) and \( \pi_{14}(G/O) \cong \mathbb{Z}/2 \oplus \mathbb{Z}/2 \); additionally, \( j_* : \pi_{15}^s \to \pi_{15}(G/O) \cong \mathbb{Z}/2 \) sends \( \eta\kappa \) to the generator of $\pi_{15}(G/O)$ \cite[Theorem~1.1.14]{ravenel}. Combining these facts, the lower part of diagram~\eqref{commutative_diagram} shows that the image of \( \eta_{14}^* : \pi_{14}(G/O) \to \pi_{15}(G/O) \) is \( \mathbb{Z}/2 \). Now, since \( \psi_* : \pi_{14}(Top/O) \to \pi_{14}(G/O) \) is injective, and the exotic sphere in \( \pi_{14}(Top/O) \) corresponds to either \( \kappa \) or \( \kappa + \sigma^2 \) in \( \pi_{14}^s \) \cite[Page~45]{kubo}, it follows from the upper part of diagram~\eqref{commutative_diagram} that \( \eta_{14}^* : \pi_{14}(Top/O) \to \pi_{15}(Top/O) \) is a monomorphism. Hence, its image is isomorphic to \( \mathbb{Z}/2 \).\\
For $k=15$, the generator $\eta\kappa \in \pi_{15}^s$ maps to zero in $\pi_{16}^s$ under $\eta_{15}^*$ \cite[Theorem 14.1]{toda}, but maps to the generator of $\pi_{15}(G/O)=\mathbb{Z}/2$ under the map $j_*$ \cite[Theorem 1.1.14]{ravenel}. Therefore, from the lower part of the commutative diagram \eqref{commutative_diagram}, we obtain that $\eta_{15}^*:\pi_{15}(G/O)\to \pi_{16}(G/O)$ is the zero map. Since $\psi_*:\pi_{15}(Top/O)\to \pi_{15}(G/O)$ is surjective and $\psi_*:\pi_{16}(Top/O)\to \pi_{16}(G/O)$ is injective, the commutativity of the upper row of the diagram then implies that the map $\eta_{15}^*:\pi_{15}(Top/O)\to \pi_{16}(Top/O)$ is also trivial.


Now we consider the case $k=16$. Note that the map $\eta_{16}^*$ sends the generator $\eta^*$ of $\pi_{16}^s$ to the generator $\eta\eta^*$ of $\pi_{17}^s$ \cite[Page 189-190]{toda}. Since $j_*(\eta^*)$ generates the torsion part of $\pi_{16}(G/O)$, and $j_*(\eta\eta^*)\neq 0$ in $\pi_{17}(G/O)$ \cite[Table A3.3]{ravenel}, it follows from the lower part of the diagram \eqref{commutative_diagram} that the restriction of \( \eta_{16}^* : \pi_{16}(G/O) \to \pi_{17}(G/O) \) to the torsion subgroup of $\pi_{16}(G/O)$ has image \( \mathbb{Z}/2 \). 
Since the map $\psi_*$ maps $\pi_{16}(Top/O)$ injectively to the torsion part of $\pi_{16}(G/O)$, the upper part of the diagram \eqref{commutative_diagram} implies that the image of \(\eta_{16}^*: \pi_{16}(Top/O) \to \pi_{17}(Top/O)\) is also \(\mathbb{Z}/2\).


For $k=17$, we note that the map $j_*: \pi_{17}^s\to \pi_{17}(G/O)$ is onto and its image is generated by $j_*(\eta\eta^*), j_*(\nu\kappa), j_*(\bar{\mu})$ \cite[Table A3.3 and Theorem 1.1.13]{ravenel}, where $\pi_{17}^s=\mathbb{Z}/2\{\eta\eta^*\}\oplus \mathbb{Z}/2\{\nu\kappa\}\oplus \mathbb{Z}/2\{\bar{\mu}\}\oplus\mathbb{Z}/2\{\eta^2\rho\}$ are the generators. Since $\eta^2\eta^*=4\nu^*, \eta \nu\kappa=0$ \cite[Theorem 14.1]{toda}, the image of the map $\eta_{17}^*:\pi_{17}^s\to \pi_{18}^s$, restricted to the subgroup $\mathbb{Z}/2\{\eta\eta^*\}\oplus\mathbb{Z}/2\{\nu\kappa\}\oplus\mathbb{Z}/2\{\bar{\mu}\}$, is $\mathbb{Z}/2\{4\nu^*\}\oplus \mathbb{Z}/2\{\eta\bar{\mu}\}$. As \( j_* : \pi_{18}^s \to \pi_{18}(G/O) \) is an isomorphism, the lower part of diagram~\eqref{commutative_diagram} implies that the image of \( \eta_{17}^* : \pi_{17}(G/O) \to \pi_{18}(G/O) \) is \( \mathbb{Z}/2 \oplus \mathbb{Z}/2 \). Since \( \psi_* : \pi_{17}(Top/O) \to \pi_{17}(G/O) \) is surjective and \( \psi_* : \pi_{18}(Top/O) \to \pi_{18}(G/O) \) is an isomorphism, the upper part of the diagram shows that \( \eta_{17}^* : \pi_{17}(Top/O) \to \pi_{18}(Top/O) \) also has image \( \mathbb{Z}/2 \oplus \mathbb{Z}/2 \).

Next, consider the case \( k = 18 \). Since \( \pi_i(Top/O) \cong \pi_i(PL/O) \) for all \( i \geq 6 \), computing the image of \( \eta_{18}^*: \pi_{18}(Top/O) \to \pi_{19}(Top/O) \) reduces to computing the image of \( \eta_{18}^*: \pi_{18}(PL/O) \to \pi_{19}(PL/O) \). This map appears in the following commutative diagram:
\[\begin{tikzcd}
	& 0 & {\pi_{18}(PL)} & {\pi_{18}(PL/O)} & 0 \\
	0 & {\pi_{19}(O)\cong \mathbb{Z}} & {\pi_{19}(PL)} & {\pi_{19}(PL/O)} & {0.}
	\arrow[from=1-2, to=1-3]
	\arrow["\cong", from=1-3, to=1-4]
	\arrow["{\eta_{18}^*}"', from=1-3, to=2-3]
	\arrow[from=1-4, to=1-5]
	\arrow["{\eta_{18}^*}", from=1-4, to=2-4]
	\arrow[from=2-1, to=2-2]
	\arrow[from=2-2, to=2-3]
	\arrow["\beta"',from=2-3, to=2-4]
	\arrow[from=2-4, to=2-5]
\end{tikzcd}\]
Here, the rows are induced by the fiber sequence \( PL/O \to BO \to BPL \); the map \( \pi_{18}(PL) \to \pi_{18}(PL/O) \) is surjective \cite[Page 291]{gw2}; and \( \pi_{18}(PL) \cong \pi_{18}(PL/O) \) since \( \pi_{18}(O) = 0 \). Moreover, \( \pi_{19}(PL) \cong \mathbb{Z} \oplus \mathbb{Z}/2 \oplus \mathbb{Z}/8 \) \cite[Theorem 1.4]{gw2}, and as noted in \cite[Page 307]{gw2}, the restriction of \( \beta: \pi_{19}(PL)\to \pi_{19}(PL/O) \) to the torsion summand is injective. Hence, by commutativity of the diagram, determining the image of \( \eta_{18}^*: \pi_{18}(PL/O) \to \pi_{19}(PL/O) \) reduces to determining the image of \( \eta_{18}^*: \pi_{18}(PL) \to \pi_{19}(PL) \). To compute this, we consider the following commutative diagram induced from the fiber sequence $PL\to G\to G/PL$.
\[\begin{tikzcd}
	& 0 & {\pi_{18}(PL)} & {\pi_{18}^s} & {\pi_{18}(G/PL)} \\
	0 & {\pi_{20}(G/PL)} & {\pi_{19}(PL)} & {\pi_{19}^s} & {0.}
	\arrow[from=1-2, to=1-3]
	\arrow["\cong", from=1-3, to=1-4]
	\arrow["{\eta_{18}^*}"', from=1-3, to=2-3]
	\arrow["0", from=1-4, to=1-5]
	\arrow["{\eta_{18}^*}", from=1-4, to=2-4]
	\arrow[from=2-1, to=2-2]
	\arrow[from=2-2, to=2-3]
	\arrow["j_{PL}"',from=2-3, to=2-4]
	\arrow[from=2-4, to=2-5]
\end{tikzcd}\]
Here, $\pi_{18}(PL)\cong \pi_{18}^s$ since $\pi_{18}(PL)\cong \pi_{18}(PL/O)$. From \cite[Theorem 14.1]{toda}, we get $\eta\circ\nu^*=0$, $\eta^2\circ \bar{\mu}=4\bar{\zeta}$ \cite[Theorem 14.1]{toda}, so the image of the map $\eta_{18}^*:\pi_{18}^s\to \pi_{19}^s$ is $\mathbb{Z}/2\{4\bar{\zeta}\}$. Since the torsion components of $\pi_{19}(PL)$ maps injectively into $\pi_{19}^s$ under $j_{PL}$ \cite[Page 307]{gw2}, the commutative diagram above implies that the map $\eta_{18}^*:\pi_{18}(PL)\to \pi_{19}(PL)$ has image $\mathbb{Z}/2$. This completes the proof for the case $k=18$.

For $k=19$, we note from \cite[Table A3.3]{ravenel} that $\pi_{19}(G/O)$ is generated by $j_*(\bar{\sigma})$. But  $\bar{\sigma}\in \langle \nu,\eta\circ\sigma,\sigma\rangle$ with zero indeterminacy \cite{toda}. Then by properties of Toda brackets \cite[Page 33]{toda}, we have
\[\eta\bar{\sigma}\in \eta \langle\nu, \eta\sigma, \sigma\rangle\subseteq \langle\eta\nu, \eta\sigma, \sigma\rangle.\]
Since $\eta\nu=0$ \cite[Theorem 14.1]{toda} and the Toda bracket $\langle\eta\nu, \eta\sigma, \sigma\rangle$ has indeterminancy zero, $\eta\bar{\sigma}=0$. Hence, by the lower part of the commutative diagram \eqref{commutative_diagram}, we have $\eta_{19}^*:\pi_{19}(G/O)\to \pi_{20}(G/O)$ is a trivial map. Since $\psi_*:\pi_{20}(Top/O)\to \pi_{20}(G/O)$ is injective, it then follows from the upper part of the diagram \eqref{commutative_diagram} that $\eta_{19}^*:\pi_{19}(Top/O)\to \pi_{20}(Top/O)$ is also the zero map.
	\end{proof}
    From the proof of the above Lemma and \cite{toda}, we observe the following:
    \begin{rmk}
The image of \(\eta^*:\pi_k(Top/O)\to \pi_{k+1}(Top/O)\) is generated by the exotic spheres corresponding to the elements \(\eta\epsilon\), \(\eta\mu\), \(\eta\kappa\), \(\eta\eta^*\), \(4\nu^*\), and \(\eta\bar{\mu}\) in \(\pi_{k+1}^s\) for \(k = 8, 9, 14, 16,\) and \(17\), respectively.
\end{rmk}

	\begin{cor}
    \indent
		\begin{itemize}
			\item[(i)] The map $(\eta^2)^*:\pi_k(Top/O)\to\pi_{k+2}(Top/O)$ is zero for $7\leq k\leq 19$, except for $k=9, 16$ and $17$.
			\item[(ii)] The image of the map $(\eta^2)^*:\pi_{k}(Top/O)\to\pi_{k+2}(Top/O)$ is $\z/2$ for $k=9, 16$ and $17$.
		\end{itemize}
	\end{cor}

Using the long exact sequence induced from the cofiber sequence $\mathbb{S}^0\xrightarrow{\times 2^r} \mathbb{S}^0\xhookrightarrow{i^{2^r}_0} M(\mathbb{Z}/{2^r})$ along $Top/O$ and Lemma \ref{lemma2.17}, we have the following result.
 	\begin{cor}
		The image of $(i^{2^r}_{3+k} \circ \eta)^*:[\Sigma^{3+k} M(\Z/2^r),Top/O]\to\pi_{4+k}(Top/O)$ is 
		\begin{itemize}
			\item[(i)] $0$ for $k= 1, 2, 3, 4, 8, 9, 10, 12, 16$,
			\item[(ii)] $\z/2$ for $k= 5, 6, 7, 11, 13, 15$,
			\item[(iii)] $\z/2\oplus\z/2$ for $k=14$.
		\end{itemize}
  where $i^{2^r}_{3+k}: \s^{3+k}\hookrightarrow \Sigma^{3+k} M(\z/{2^r})$ is the inclusion map.
	\end{cor}
 
 We are now ready to establish the computations of $I_c(M\times\s^k)$ for $3\leq k\leq 16.$
	\begin{theorem}\label{concorM^4}
		Let $M$ be a closed, oriented, smooth $4$-manifold.
		\begin{enumerate}
			\item[(i)] If $M$ is a spin manifold, then $I_c(M\times\s^k)=0$ for $k\geq 3.$
			
			\item[(ii)] If $M$ is a non-spin manifold, then 
			\begin{itemize}
				\item[(a)] $I_c(M\times\s^k)=0$ for $k=3,4,8,9,10,12,16$, 
				\item[(b)] $I_c(M\times\s^k)=\mathbb{Z}/2$ for $k=5,6,7,11,13,15$,
				\item[(c)] $I_c(M\times\s^{k})=\mathbb{Z}/2\oplus\mathbb{Z}/2$ for $k=14$.
              
			\end{itemize}
		\end{enumerate}
		
	\end{theorem}
	\begin{proof}
The proof is obtained by applying the stable splittings (\ref{spinM^4}), and (\ref{nonspinM^4}) for $M,$ in combination with Lemma \ref{concordanceinertia}, and Lemma \ref{lemma2.17}.
	\end{proof}
	
	\begin{prop}\label{inerM^4}
          Let $M$ be a closed, smooth, oriented, non-spin manifold of dimension $4$. Then $\z/2 \subseteq I_c(M\times\s^k)$ for $k=17,18, 8n-2(n\geq 3).$ 
	\end{prop}
    \begin{proof}
From \cite[Table A3.3]{ravenel} and \cite[Theorem 1.3, and 1.4]{JFADAMS66}, it can be obtained that for $k=17,18$, or $8n-2 (n\geq 3)$, there exists an element of order 2, denoted as $x_{3+k}$, in $\pi_{3+k}^s/{\text{Im}(J_{3+k})}$ such that $\eta \circ x_{3+k}$ is non-zero in $\pi_{4+k}^s/{\text{Im}(J_{4+k})}$. The existence of such elements, along with the surjectivity of the induced map  $\psi^*:\pi_{3+k}(Top/O)\to \pi_{3+k}^s/{\text{Im}(J_{3+k})}$, implies that the image of $\eta^*:\pi_{3+k}(Top/O)\to\pi_{4+k}(Top/O)$ includes an exotic $(4+k)$-sphere corresponding to the element $\eta \circ x_{3+k}$. The proposition then follows from the stable splitting (\ref{nonspinM^4}) and Lemma \ref{concordanceinertia}.
    \end{proof}
 We now proceed to calculate $\mathcal{C}(M\times\mathbb{S}^k)$. It follows from Corollary \ref{concorMtimesSk} that it is sufficient to compute the summand $[\Sigma^k M,Top/O]$, which can be obtained from either the stable splitting (\ref{spinM^4}) or (\ref{nonspinM^4}) by examining the same computations for each component. Therefore, we first compute the component $[\Sigma^k \mathbb{C}P^2, Top/O]$. Consider the cofiber sequence for the map 
 
 \begin{equation}\label{coficp^2}
		\mathbb{S}^3\xrightarrow{\eta_2} \mathbb{S}^2\xhookrightarrow{i}\mathbb{C}P^2\xrightarrow{f_{\mathbb{C}P^2}}\Sigma\mathbb{S}^3\xrightarrow{\Sigma \eta_2}\Sigma\mathbb{S}^2\hookrightarrow\cdots
	\end{equation}
 This sequence induces a long exact sequence
	\begin{equation}\label{eqncp^2}
		\cdots\to\pi_{3+k}(X)\xrightarrow{\eta^*}\pi_{4+k}(X)\xrightarrow{(\Sigma^kf_{\mathbb{C}P^2})^*}[\Sigma^k\mathbb{C}P^2,X]\xrightarrow{(\Sigma^ki)^*}\pi_{2+k}(X)\xrightarrow{\eta^*}\pi_{3+k}(X)\to\cdots
	\end{equation} where $\Sigma^k \eta_2=\eta$ for $k\geq1$ \cite{toda} and $X$ is any infinite loop space.

The following lemma will be useful in the computation of $[\Sigma^k \mathbb{C}P^2, Top/O]$.

\begin{lemma}\label{etaG/O}
   The image of the map $\eta^*: \pi_8(G/O)\to \pi_9(G/O)$ is $\mathbb{Z}/2\oplus\mathbb{Z}/2$.
\end{lemma}
\begin{proof}
  Since $\pi_9(G/O)\cong \mathbb{Z}/2\oplus\mathbb{Z}/2$, it suffices to work locally at prime $2$. We note from \cite[Theorem 5.18]{milgram} that
    \[G/O_{(2)}\simeq BSO_{(2)}\times cok(J_{(2)}).\] Here $\pi_8(cok(J_{(2)}))$ is the cokernel of the $J$-homomorphism $J:\pi_8(O)\to \pi_8^s$ and is equal to $\mathbb{Z}/2\{[\bar{\nu}]\}$ \cite[Theorem 1.1.14]{ravenel}. Moreover, $\pi_9(cok(J_{(2)}))$ is isomorphic to an index two summand of the cokernel of the $J$-homorphism $J:\pi_9(O)\to \pi_9^s$ \cite[Remark 11.43]{Luck_surgery} and is given by $\mathbb{Z}/2\{[\nu^3]\}$ (see Lemma \ref{9_coker}). Using the above decomposition, we obtain the following commutative diagram:
    \[
    \begin{tikzcd}
	{\pi_8(G/O_{(2)})} && {\pi_9(G/O_{(2)})} \\
	\\
	{\pi_8(BSO_{(2)})\oplus \pi_8(cok(J_{(2)}))} && {\pi_9(BSO_{(2)})\oplus \pi_9(cok(J_{(2)})).}
	\arrow["{\eta^*}", from=1-1, to=1-3]
	\arrow["{\cong }"', from=1-1, to=3-1]
	\arrow["\cong", from=1-3, to=3-3]
	\arrow["{\eta^*\oplus \eta^*}"', from=3-1, to=3-3]
\end{tikzcd}\]
Since $[\bar{\nu}]\in \pi_8(cok(J_{(2)}))$, $[\nu^3]\in \pi_9(cok(J_{(2)}))$ and $\eta\circ \bar{\nu}=\nu^3$ \cite[Theorem 14.1]{toda}, the map $\eta^*:\pi_8(cok(J_{(2)}))\to \pi_9(cok(J_{(2)}))$ has image $\mathbb{Z}/2$. Moreover from \cite[Page 29]{bel}, the map $\eta^*: \pi_8(BSO_{(2)})\to \pi_9(BSO_{(2)})$ is surjective. Hence, by the commutativity of the above diagram, the image $\eta^*:\pi_8(G/O_{(2)})\to \pi_9(G/O_{(2)})$ is $\mathbb{Z}/2\oplus\mathbb{Z}/2$.
\end{proof}
\begin{lemma}\label{9_coker}
    The generator of $\pi_9(cok(J_{(2)}))$ is $[\nu^3]$.
\end{lemma}
\begin{proof}
We first show that the map $$\eta^*: \pi_9(BSO_{(2)})\to \pi_{10}(BSO_{(2)})$$ is an isomorphism. To show this, we consider the following commutative diagram:
\[\begin{tikzcd}
	{\pi_8^s} && {\pi_9^s} \\
	{\mathbb{Z}/2\cong \pi_9(BSO_{(2)})} && {\pi_{10}(BSO_{(2)})\cong \mathbb{Z}/2,}
	\arrow["{\eta^*}", from=1-1, to=1-3]
	\arrow["J", from=2-1, to=1-1]
	\arrow["{\eta^*}"', from=2-1, to=2-3]
	\arrow["J"', from=2-3, to=1-3]
\end{tikzcd}\]
where both maps $\eta^*: \pi_8^s \to \pi_9^s$ and $J: \pi_9(BSO_{(2)}) \to \pi_8^s$ are injective by \cite[pp.~189--190]{toda} and \cite[Theorem~1.1.13]{ravenel}, respectively. Hence, it follows from the above diagram that $\eta^*: \pi_9(BSO_{(2)}) \to \pi_{10}(BSO_{(2)})$ is an isomorphism. 

Next, we consider the map \[\eta^*: \pi_9(G/O_{(2)})\to \pi_{10}(G/O_{(2)}).\] Since $\psi_*: \pi_{10}(Top/O)\to \pi_{10}(G/O)$ is an isomorphism, $\psi_*: \pi_9(Top/O)\to \pi_9(G/O)$ is surjective, and $\eta^*: \pi_9(Top/O_{(2)})\to \pi_{10}(Top/O_{(2)})$ is surjective by Lemma \ref{lemma2.17}, it follows from the diagram \eqref{commutative_diagram} that $$\eta^*: \pi_9(G/O_{(2)})\to \pi_{10}(G/O_{(2)})$$ is surjective.

From \cite[Theorem 1.1.14]{ravenel}, we have
\[
\pi_9(G/O_{(2)}) \cong \mathbb{Z}/2\{[\nu^3]\} \oplus \mathbb{Z}/2\{[\mu]\}, \quad
\pi_{10}(G/O_{(2)}) \cong \mathbb{Z}/2\{[\eta \circ \mu]\}.
\] 
Also, by \cite[Remark 11.43]{Luck_surgery}, we have
\[
\pi_{10}(cok(J_{(2)})) = 0.
\]

We now show that $[\nu^3] \in \pi_9(cok(J_{(2)})) \cong \mathbb{Z}_2$. Consider the following commutative diagram:
\[\begin{tikzcd}
	{\mathbb{Z}/2\{[\nu^3]\}\oplus \mathbb{Z}/2\{[\mu]\}\cong \pi_9(G/O_{(2)})} && {\pi_{10}(G/O_{(2)})\cong \mathbb{Z}/2\{[\eta\circ \mu]\}} & 0 \\
	{\pi_9(BSO_{(2)})\oplus \pi_9(cok(J_{(2)}))} && {\pi_{10}(BSO_{(2)}).}
	\arrow["{\eta^*}", from=1-1, to=1-3]
	\arrow["\cong"', from=1-1, to=2-1]
	\arrow[from=1-3, to=1-4]
	\arrow["\cong", from=1-3, to=2-3]
	\arrow["{\eta^*\oplus \eta^*}"', from=2-1, to=2-3]
\end{tikzcd}\]
Here the map $\eta_*: \pi_9(cok(J_{(2)})) \to \pi_{10}(cok(J_{(2)}))$ is trivial. Since $\eta_* : \pi_9(BSO_{(2)}) \to \pi_{10}(BSO_{(2)})$ is an isomorphism, the kernel of
\[
\eta_* \oplus \eta_*: \pi_9(BSO_{(2)}) \oplus \pi_9(cok(J_{(2)})) \longrightarrow \pi_{10}(BSO_{(2)})
\]
is precisely the summand $\pi_9(cok(J_{(2)}))$.


On the other hand, since $\eta^*: \pi_9(G/O_{(2)})\to \pi_{10}(G/O_{(2)})$ is surjective and its image is generated by $[\eta\circ \mu]$, both the elements $[\mu]$ and $[\mu]\oplus[\nu^3]$ of $\pi_9(G/O_{(2)})$ map non-trivially under $\eta^*$. Therefore, 
\[
\mathit{ker}\big(\eta^*\colon \pi_9(G/O_{(2)})\to \pi_{10}(G/O_{(2)})\big)
= \mathbb{Z}/2\{[\nu^3]\}.
\]


From these two observations and the commutative diagram, it follows that \[[\nu^3]\in \pi_9(cok(J_{(2)})).\]

This completes the proof.

\end{proof}

	\begin{prop}\label{prop3.6}
		\hspace{2em}
		\begin{enumerate}
			\item[(1)] $[\Sigma^2\mathbb{C}P^2,Top/O]\cong0.$
			\item[(2)] $[\Sigma^3\mathbb{C}P^2,Top/O]\cong\Theta_7.$
                 \item[(3)] $[\Sigma^4\mathbb{C}P^2,Top/O]\cong \Theta_8.$
			\item[(4)] There is a non-split short exact sequence $$0 \to \mathbb{Z}/2\oplus\mathbb{Z}/2\to [\Sigma^5\mathbb{C}P^2,Top/O]\to \Theta_7\to 0,$$ where $\mathbb{Z}/2\oplus\mathbb{Z}/2\subset \Theta_9.$
			\item[(5)] $[\Sigma^6\mathbb{C}P^2,Top/O]\cong\z/3\subset \Theta_{10}.$
			\item[(6)] There is a short exact sequence $$0\to\z/{496}\to[\Sigma^7\mathbb{C}P^2,Top/O]\to \z/2\oplus\z/2\to0,$$ where $\mathbb{Z}/{496}\subset \Theta_{11}$ and $\mathbb{Z}/2\oplus\mathbb{Z}/2\subset \Theta_9.$
			\item[(7)] $[\Sigma^8\mathbb{C}P^2,Top/O]\cong\z/3\subset \Theta_{10}.$
			\item[(8)] $[\Sigma^9\mathbb{C}P^2,Top/O]\cong\Theta_{13}\oplus\Theta_{11}.$
			\item[(9)] $[\Sigma^{10}\mathbb{C}P^2,Top/O]\cong\Theta_{14}.$
		\end{enumerate}
	\end{prop}
	\begin{proof}
 
In this proof, we make use of the long exact sequence (\ref{eqncp^2}) for $X$ being either $Top/O$, $G/O$, or $G/Top$. We also use the long exact sequence derived from the fiber sequence $\Omega(G/Top)\xrightarrow{\omega}Top/O\xrightarrow{\psi}G/O\xrightarrow{\phi}G/Top.$ 


Statements {(1)} to {(9)}, except {(4)}, are direct consequences of the long exact sequence (\ref{eqncp^2}) and Lemma \ref{lemma2.17}. For statement {(4)}, we obtain the following short exact sequence from the long exact sequence \eqref{eqncp^2} and Lemma \ref{lemma2.17} :
\begin{equation}\label{sigma5cp2}
0 \to \z/2 \oplus \z/2 \to [\Sigma^5 \mathbb{C}P^2,Top/O] \to \Theta_7 \to 0.
\end{equation}


To show that this short exact sequence does not split, we consider the following commutative diagram where the columns are induced from $\cdots\to \Omega(G/Top)\xrightarrow{\omega} Top/O\xrightarrow{\psi}G/O\xrightarrow{\phi}G/Top$ and the rows are induced from the cofiber sequence \eqref{coficp^2}.
            \begin{center}
			\begin{tikzcd}
				& 0\arrow[d]\\
				0 \arrow[r] \arrow[d] &{\z/2} \arrow[r,"(\Sigma^6 f_{\mathbb{C}P^2})^*"] \arrow[d,"\omega_*"'] & {[\Sigma^6 \mathbb{C}P^2,G/Top]} \arrow[r,"(\Sigma^6 i)^*"] \arrow[d,"\omega_*"] & {\z}\arrow[r] \arrow[d,"\omega_*"] & 0 \arrow[d]\\
				{\z/2} \arrow[r,"\eta^*"] \arrow[d] & {\bigoplus\limits_{i=1}^3\z/2} \arrow[r,"(\Sigma^5 f_{\mathbb{C}P^2})^*"] \arrow[d,"\psi_*"'] & {[\Sigma^5 \mathbb{C}P^2,Top/O]} \arrow[r,"(\Sigma^5 i)^*"] \arrow[d,"\psi_*"']& {\z/28} \arrow[r] \arrow[d] & {0}\\
				{\z\oplus \z/2} \arrow[r,"\eta^*"] & {\z/2 \oplus \z/2} \arrow[r,"(\Sigma^5 f_{\mathbb{C}P^2})^*"]\arrow[d]& {[\Sigma^5 \mathbb{C}P^2,G/O]} \arrow[r]\arrow[d] & 0\\
				& 0 & {0.}
			\end{tikzcd}
		\end{center} 
Here the top row splits at $[\Sigma^{6}\mathbb{C}P^2, G/Top]$ and $\psi_*: [\Sigma^5 \mathbb{C}P^2, Top/O] \to [\Sigma^5 \mathbb{C}P^2, G/O]$ is surjective as  $[\Sigma^5 \mathbb{C}P^2, G/Top] \cong 0$. Since \( \pi_7(G/O) = 0 \) and the map \( \eta^* : \pi_8(G/O) \to \pi_9(G/O) \) is surjective by Lemma~\ref{etaG/O}, it follows from the bottom row of the diagram that \( [\Sigma^5 \mathbb{C}P^2, G/O] = 0 \). 
Consequently, we deduce that the map $\omega_\ast : [\Sigma^6 \C P^2, G/Top] \to [\Sigma^5 \C P^2, Top/O]$ is also surjective. By combining this with the fact $[\Sigma^6 \C P^2, G/Top]=\mathbb{Z}\oplus \mathbb{Z}/2$ and the commutativity of the top rectangle,  we conclude that $[\Sigma^5 \C P^2, Top/O] \cong \Z/2 \oplus \Z/56$. This confirms that the short exact sequence \eqref{sigma5cp2} does not split, thereby completing the proof of Statement {(4)}. 
This concludes the proof of the proposition.
\end{proof}


 We frequently use the long exact sequence 
 \begin{equation}\label{longMoore}
     \cdots\xrightarrow{\times p^r}\pi_{4+k}(Top/O)\xrightarrow{(q^{p^r}_{4+k})^*}[\Sigma^{3+k}M(\mathbb{Z}/{p^r}),Top/O]\xrightarrow{(i^{p^r}_{3+k})^*}\pi_{3+k}(Top/O)\xrightarrow{\times p^r} \cdots
 \end{equation}
 induced from the cofiber sequence 
 \begin{equation}\label{Moore}
    \mathbb{S}^{3+k}\xrightarrow{\times p^r}\mathbb{S}^{3+k}\xhookrightarrow{i^{p^r}_{3+k}}\Sigma^{3+k} M(\mathbb{Z}/{p^r})\xrightarrow{q^{p^r}_{4+k}}\mathbb{S}^{4+k}\cdots
 \end{equation}
  throughout the paper.

By using the group structure of $\pi_m(Top/O)$ $(1\leq m\leq 18)$ localized at the primes and the long exact sequence \eqref{longMoore}, we get the following lemma:
\begin{lemma}\label{moorespacecal}
    Given any positive integer $r$, the following holds:
    \begin{itemize}
        \item[(a)] $[\Sigma^k M(\mathbb{Z}/{3^r}),Top/O]= 0,$ for $1\leq k\leq  17$ with $k\neq 9,10,12,$ and $13.$
        \item[(b)] $[\Sigma^k M(\mathbb{Z}/3^r),Top/O]$ equals $\z/3,$ for $k=9,10,12,$ and $13.$ 
        \item[(c)] $[\Sigma^k M(\mathbb{Z}/{7^r}),Top/O]=0$ for $1\leq k\leq 17$ with $k\neq 6,7.$
        \item[(d)] $[\Sigma^k M(\mathbb{Z}/{7^r}),Top/O]$ equals $\z/7,$ for $k=6$ and $7.$
        \item[(e)] $[\Sigma^{11}M(\mathbb{Z}/{31^r}), Top/O]=\mathbb{Z}/{31}$. 
    \end{itemize}
\end{lemma}
\begin{lemma}\label{moore_space_cal_2}
Let $r$ be any positive integer. Then the following holds:
    \begin{itemize}
        \item[(a)] $[\Sigma^6 M(\Z/{2^r}),Top/O] = \begin{cases}\Z/2 &\mbox{if $r=1,$}\\ \mathbb{Z}/4 &\mbox{if $r\geq 2$.}\end{cases}$
        \item[(b)] $[\Sigma^7 M(\Z/{2^r}),Top/O] = \begin{cases}\Z/2\oplus\mathbb{Z}/2 &\mbox{if $r=1,$}\\ \mathbb{Z}/2\oplus\mathbb{Z}/4 &\mbox{if $r\geq 2$.}\end{cases}$
            \item[(c)] $[\Sigma^8 M(\Z/2), Top/O]=\mathbb{Z}/4\oplus\mathbb{Z}/2\oplus\mathbb{Z}/2$.
            \item[(d)] $[\Sigma^9 M(\Z/{2^r}),Top/O] = \begin{cases}\Z/4\oplus\Z/2\oplus\Z/2 &\mbox{if $r=1,$}\\ \bigoplus\limits_{i=1}^4\mathbb{Z}/2 &\mbox{if $r\geq 2$.}\end{cases}$
        \item[(e)] $[\Sigma^{10}M(\mathbb{Z}/{2}), Top/O]=\mathbb{Z}/2\oplus \mathbb{Z}/2$.
        \item[(f)] $[\Sigma^{11} M(\Z/{2^r}),Top/O] = \begin{cases}\Z/{2^r} &\mbox{if $1\leq r\leq 5,$}\\ \mathbb{Z}/{2^5} &\mbox{if $r\geq 6$.}\end{cases}$
    \end{itemize}
\end{lemma}
\begin{proof}
Since $\pi_6(Top/O)$ and $\pi_{12}(Top/O)$ are both trivial, then from the long exact sequence \eqref{longMoore}, we get 
    \[[\Sigma^6 M(\mathbb{Z}/{2^r}), Top/O]= \mathit{Coker}\left(\pi_{7}(Top/O)\cong \mathbb{Z}/28\xrightarrow{\times 2^r}\pi_7(Top/O)\cong \mathbb{Z}/28\right)\] and
    \[[\Sigma^{11}M(\mathbb{Z}/{2^r}), Top/O]=\mathit{ker}\left(\pi_{11}(Top/O)\cong \mathbb{Z}/{992}\xrightarrow{\times 2^r}\pi_{11}(Top/O)\cong \mathbb{Z}/{992}\right).\]
This gives the results {(a)} and {(f)}.

We note from \eqref{longMoore} that $[\Sigma^8 M(\mathbb{Z}/2), Top/O]$ fits into the following short exact sequence
\[0\to \pi_{9}(Top/O)=\bigoplus\limits_{i=1}^3\mathbb{Z}/2\to [\Sigma^8 M(\mathbb{Z}/2), Top/O]\to \pi_8(Top/O)=\mathbb{Z}/2\to 0. \]
This implies that $[\Sigma^8 M(\mathbb{Z}/2), Top/O]$ is either $\bigoplus\limits_{i=1}^4 \mathbb{Z}/2$ or $\mathbb{Z}/4\oplus\mathbb{Z}/2\oplus\mathbb{Z}/2$. Since $\eta^*: \pi_8(Top/O)\to \pi_9(Top/O)$ has image $\mathbb{Z}/2$ by Lemma \ref{lemma2.17}, it follows from \cite[Lemma 2.3]{DCTSWS2018} that $[\Sigma^8 M(\mathbb{Z}/2), Top/O]=\mathbb{Z}/4\oplus\mathbb{Z}/2\oplus\mathbb{Z}/2$. 

For $k=7$, $p=2$, and $r=1$, we obtain from \eqref{longMoore} the short exact sequence  
\[
0 \longrightarrow \mathbb{Z}/2 \xrightarrow{(q_{11}^2)^*} [\Sigma^{10}M(\mathbb{Z}/2), Top/O] \xrightarrow{(i_{10}^2)^*} \mathbb{Z}/2 \longrightarrow 0,
\]  
where the group $\mathbb{Z}/2$ on the left is the cokernel of the map  
$\pi_{11}(Top/O)\xrightarrow{\times 2} \pi_{11}(Top/O)$,  
and the $\mathbb{Z}/2$ on the right is the kernel of the map  
$\pi_{10}(Top/O)\xrightarrow{\times 2}\pi_{10}(Top/O)$.  

We claim that this sequence splits. Let $\bar{x}\in [\Sigma^{10}M(\mathbb{Z}/2),Top/O]$ be an element such that $(i_{10}^2)^*(\bar{x})=x$, where $x$ is the generator of $\mathit{ker}(\pi_{10}(Top/O)\xrightarrow{\times 2}\pi_{10}(Top/O))$. By Lemma \ref{lemma2.17}, the image of the map $\eta^*: \pi_{10}(Top/O)\to \pi_{11}(Top/O)$ is $\mathbb{Z}/2$, which is contained in the image of $\pi_{11}(Top/O)\xrightarrow{\times 2} \pi_{11}(Top/O)$. Hence, by \cite[Lemma 2.3]{DCTSWS2018}, we have  
\[
2\bar{x} = [\mu \circ \eta] = 0 \quad \text{in } \pi_{11}(Top/O)/\operatorname{Im}(\pi_{11}(Top/O)\xrightarrow{\times 2}\pi_{11}(Top/O)).
\]  
This shows that $[\Sigma^{10}M(\mathbb{Z}/2), Top/O]\cong \mathbb{Z}/2\oplus \mathbb{Z}/2$.

(b) and~(d) will be established in the course of the proof of Lemma~\ref{etatilde} for the cases \( k = 4 \) and \( k = 6 \).
\end{proof}
Let $l_p$ represent the number of $p$-torsion summands of $H_2(M;\mathbb{Z})$ in \eqref{hom}. Also, let $l_p^i$ denote the number of $\mathbb{Z}/{p^i}$-summands in $H_2(M;\mathbb{Z})$ in \eqref{hom}.
 
      \begin{prop}\label{cM^4}
          Let $M$ be a $4$-dimensional closed, oriented manifold with homology as in (\ref{hom}). Then
          \begin{itemize}
              \item[(i)]$[\Sigma M,Top/O]\cong\bigoplus\limits_{j=1}^{2 l_{2}}\pi_3(Top/O) \oplus\bigoplus\limits_{l=1}^d \pi_3(Top/O).$
			\item[(ii)]$[\Sigma^2 M,Top/O]\cong \bigoplus\limits_{i=1}^m\pi_3(Top/O)\oplus\bigoplus\limits_{j=1}^{l_2}\pi_3(Top/O).$
			\item[(iii)]$[\Sigma^3 M,Top/O]\cong\Theta_7.$
               \item[(iv)] $[\Sigma^4 M,Top/O]\cong \Theta_8\oplus\bigoplus\limits_{i=1}^m\Theta_7\oplus\bigoplus\limits_{j=1}^{l_2^1}\mathbb{Z}/2\oplus\bigoplus\limits_{j=1}^{l_2-l_2^1}\mathbb{Z}/4\oplus\bigoplus\limits_{j=1}^{l_7} \mathbb{Z}/7.$
             \item[(v)]  \begin{itemize}
                 \item[(a)]  $[\Sigma^5 M,Top/O]\cong\Theta_9\oplus \bigoplus\limits_{i=1}^m\Theta_8\oplus \bigoplus\limits_{j=1}^{l_2^1}\mathbb{Z}/2\oplus\bigoplus\limits_{j=1}^{l_2-l_2^1}\mathbb{Z}/4\oplus \bigoplus\limits_{j=1}^{2 l_7} \mathbb{Z}/7 \\
                 \oplus \bigoplus\limits_{j=1}^{l_2^1}(\mathbb{Z}/2\oplus\mathbb{Z}/2)\oplus\bigoplus\limits_{j=1}^{l_2-l_2^1}(\mathbb{Z}/2\oplus\mathbb{Z}/4) \oplus\bigoplus\limits_{l=1}^d\Theta_7,$ if $M$ is a spin manifold.
                 \item[(b)] $[\Sigma^5 M,Top/O]\cong\z/{56}\oplus\z/2\oplus \bigoplus\limits_{i=1}^m\Theta_8
                 \oplus \bigoplus\limits_{j=1}^{l_2^1}\mathbb{Z}/2\oplus\bigoplus\limits_{j=1}^{l_2-l_2^1}\mathbb{Z}/4\oplus \bigoplus\limits_{j=1}^{2 l_7} \mathbb{Z}/7 \\ \oplus \bigoplus\limits_{j=1}^{l_2^1}(\mathbb{Z}/2\oplus\mathbb{Z}/2)\oplus\bigoplus\limits_{j=1}^{l_2-l_2^1}(\mathbb{Z}/2\oplus\mathbb{Z}/4) \oplus\bigoplus\limits_{l=1}^{d-1}\Theta_7,$ if $M$ is a non-spin manifold.
             \end{itemize}
          \end{itemize}
      \end{prop}
\begin{proof}
The computations of $[\Sigma^k M,Top/O]$ for $1\leq k \leq 5$ can be immediately obtained from the stable splittings \eqref{spinM^4} and \eqref{nonspinM^4} by applying Proposition \ref{prop3.6}, Lemmas \ref{moorespacecal} and \ref{moore_space_cal_2}.
\end{proof}

Similarly, using Proposition~\ref{prop3.6} together with Lemmas~\ref{moorespacecal} and~\ref{moore_space_cal_2}, one can partially compute the groups $[\Sigma^k M, Top/O]$ for $5 \leq k \leq 10$, based on the stable decompositions of $M$ given in \eqref{spinM^4} and \eqref{nonspinM^4}.

	
	\section{Smooth structures on \texorpdfstring{$M^5\times\mathbb{S}^k$}{}}\label{5dim}
	
	Let $M$ be a simply connected, closed, oriented $5$-dimensional smooth manifold. Note that by using Poincar\'e\xspace duality and the Universal coefficient theorem, $M$ has the following homology group
	\begin{equation}\label{homologyM^5}
		H_l(M;\z) =
		\begin{cases}
			\mathbb{Z} & l=0,5\\
			\mathbb{Z}^{d} \oplus\bigoplus\limits_{i=1}^t \z/{p_i^{r_i}} & l=2\\
			\mathbb{Z}^{d} & l=3\\
			0 & \text{otherwise,}\\
		\end{cases}       
	\end{equation} where \( d, t \) are non-negative integers, \( r_j \) (for \( 1 \leq j \leq t \)) are positive integers, and \( p_i \) are prime numbers.

	\subsection{The stable homotopy type of a simply connected \texorpdfstring{$5$}{}-manifold}\label{stabletypeM^5}
	We may now consider the minimal cell structure on $M$ \cite[Propositon 4C.1]{hatcher}, for which the stable homotopy type of the $3$-skeleton $M^{(3)}$ of $M$ is given by 
	\begin{equation*}
M^{(3)}\simeq\bigvee_{i=1}^d\left(\s^2\vee\s^3\right)\vee\bigvee_{j=1}^t \Sigma^2 M(\Z/p_j^{r_j}).
	\end{equation*}
	Therefore, $M\simeq M^{(3)}\cup_{g}\mathbb{D}^5$ where $g:\s^4\to M_{(3)}$ is the attaching map of the top cell. Hence, we have a cofiber sequence
	\begin{equation*}
		\s^4\xrightarrow{g} M^{(3)}\xhookrightarrow{i} M\xrightarrow{f_{M}}\s^5\to\cdots
	\end{equation*} 
	of spectra. We divide the possibilities into two cases according to whether $M$ is spin or not. As in the $4$-dimensional case, we observe that the stable attaching map is $0$ if $M$ is spin. 
	\begin{prop}\label{5manspin}
		Let $M$ be a simply connected spin $5$-manifold with homology as in \eqref{homologyM^5}, then, we have a stable equivalence 
		\[ M\simeq \s^5 \vee \bigvee_{i=1}^d\left(\s^2\vee\s^3\right)\vee\bigvee_{j=1}^t \Sigma^2 M(\Z/p_j^{r_j}). \]
	\end{prop} 
	
	\begin{proof}
		If $M$ is a spin manifold, then from the results of Smale and Barden we see that $M$ is completely classified as a connected sum of some atomic pieces. These atomic pieces are $\s^2\times\s^3, \s^5$ and $M_k$ with $H_2(M_k)=\z/k\oplus\z/k, k\neq 1$ and $\infty$ \cite[Page-373]{barden}. We have to show that $\eta^2$ does not appear in the attaching map in each case, and this will imply the proposition. This is clear for $\s^2\times \s^3$ and $\s^5$. We now prove this for $M_k$ from an explicit cell complex construction of it following  \cite[Page-373]{barden}.
		
		Let $N$ be the manifold $(\mathbb{S}^2 \times \mathbb{S}^2)\# (\mathbb{S}^2 \times \mathbb{S}^2)$. Write $N = \partial D_i$ for $i=1,2$ with each 
		\[D_i\cong \mathbb{D}^3 \times \mathbb{S}^2 \setminus (\mathbb{D}^2 \times \mathbb{D}^2) \cup_{\partial \mathbb{D}^4 \times \mathbb{D}^1} \mathbb{D}^3 \times \mathbb{S}^2 \setminus (\mathbb{D}^2 \times \mathbb{D}^2).\] 
		Now consider $\phi : N \to N$ which is described by the matrix $A_k$ on the $2$-skeleton. 
		\[ A_k = \begin{bmatrix} 1 & 0 &0 &0 \\
			0 & 1 & k & 0\\ 
			0 & 0 & 1 & 0 \\
			- k & 0 & 0 & 1  \end{bmatrix}.\]
The manifold $M_k = D_1 \cup_{\phi : N \to N} D_2 $. This manifold corresponds to the manifold $M_n$ in Class IV of the classification in \cite[\S 9]{Sto82}, and we readily see that the attaching map of the top cell is trivial in stable homotopy. 
		This completes the proof.

	\end{proof}
	
	It now remains to consider the case where $M$ is not spin. This means that we have a non-trivial second Stiefel Whitney class, which by Wu's formula implies that $\eta$ appears in the attaching map of the top cell. If the map $\eta$ hits the $\s^3$-factor as in the $4$-dimensional case, it cones off to give a copy of the suspension of $\C P^2$. However $\eta$ may also map via $\tilde{\eta}_{2^s}$ for some $s\in \mathbb{N}$ onto the Moore spectrum part. If in addition, in the expression of the attaching map 
\[ \mathbb{S}^4 \to  \bigvee_{i=1}^{d}\left(\s^2\vee\s^3\right) \vee\bigvee_{j=1}^t \Sigma^2 M(\Z/p_j^{r_j}) \]
there is a factor of $\s^2$ which receives a map $\eta^2$, we may change the homotopy class of $\s^3 \to M^{(4)}$ or $\Sigma^2 M(\Z/2^s) \to M^{(4)}$ to make that factor $0$. That is, we use 
\[ \s^4 \stackrel{\eta_1 + \eta^2_2}{\to} \s^3\vee \s^2 \simeq \Big[\s^4 \stackrel{\eta_1+0}{\to} \s^3\vee \s^2 \stackrel{(id+\eta)\vee id}{\to} \s^3\vee \s^2\Big],\] 
and
\[ \s^4 \stackrel{\tilde{\eta}_1 + \eta^2_2}{\to} \Sigma^2 M(\Z/2^s) \vee \s^2 \simeq \Big[\s^4 \stackrel{\tilde{\eta}_1+0}{\to} \Sigma^2 M(\Z/2^s) \vee \s^2 \stackrel{(id+\hat{\eta})\vee id}{\to} \Sigma^2 M(\Z/2^s) \vee \s^2\Big],\] 
where the map $\hat{\eta}= \eta \circ \hat{q}$ for $\hat{q} : \Sigma^2 M(\Z/2^s) \to \s^3 $ the usual quotient map. In the case both $\eta$ occurs on a $\s^3$-factor and a $\tilde{\eta}$ occurs on a $\Sigma^2 M(\Z/2^s)$-factor, we may make the first factor $0$ via 
\[ \s^4 \stackrel{\eta_1 + \tilde{\eta}_2}{\to} \s^3\vee \Sigma^2 M(\Z/2^s) \simeq \Big[\s^4 \stackrel{0+\tilde{\eta}_2}{\to} \s^3\vee \Sigma^2 M(\Z/2^s) \stackrel{id\vee (\hat{q}+id)}{\to} \s^3\vee \Sigma^2 M(\mathbb{Z}/{2^s})\Big].\] 
The case where two $\tilde{\eta}$s occur, we have $\lambda_{r,s} : M(\Z/2^r) \to M(\Z/2^s)$ for $r\leq s$ such that $\lambda_{r,s} \circ \tilde{\eta}_{2^r} = \tilde{\eta}_{2^s}$, which may be used to simplify the attaching map as above. 
Thus, we may have two cases both of which may occur as demonstrated in the examples of \cite[Page-373]{barden}. Thus, if $M$ is not spin, the stable homotopy type is either 

 \begin{equation}\label{5non-spin1}
     M\simeq \Sigma \C P^2 \vee \bigvee_{i=1}^{d-1}\left(\s^2\vee\s^3\right)\vee \s^2 \vee\bigvee_{j=1}^t \Sigma^2 M(\Z/p_j^{r_j})
 \end{equation}
 \begin{center}
     or
 \end{center}
 \begin{equation}\label{5non-spin2}
      M\simeq Cone(\tilde{\eta}_{2^r}) \vee \bigvee_{i=1}^d\left(\s^2\vee\s^3\right)\vee\bigvee_{j=2}^t \Sigma^2 M(\Z/p_j^{r_j})
 \end{equation}
where $p_1^r=2^r$ is the factor which supported the $\tilde{\eta}_{2^r}$ map. The value of $r$ denotes the smallest power of $2$ present in the Moore spectrum segment.
	
\subsection{Concordance classes of \texorpdfstring{$M\times\s^k$}{}}
   It follows from Lemma \ref{concordanceinertia} that the concordance inertia group $I_c(M\times \s^k)$ can be identified with the trivial group, or \\ $Im\left(\pi_{4+k}(Top/O)\xrightarrow{\eta^*} \pi_{5+k}(Top/O)\right),$ or $Im\left([\Sigma^{3+k}M(\mathbb{Z}/{2^r}),Top/O]\xrightarrow{(\Tilde{\eta}_{2^r})^*} \pi_{5+k}(Top/O)\right),$ based on the stable decomposition given in Proposition \ref{5manspin} and \eqref{5non-spin1}, \eqref{5non-spin2}. Note that the image of $\eta^*$ is given in Lemma \ref{lemma2.17}. So we compute the image of  $(\tilde{\eta}_{2^r})^*: [\Sigma^{3+k}M(\mathbb{Z}/{2^r}),Top/O]\to \pi_{5+k}(Top/O)$ given in the following lemma. 
	
	\begin{lemma}\label{etatilde}\noindent 
 The following statements hold.
		\begin{itemize}
			\item[(a)] For each $r\geq 1$, the image of the map $(\tilde{\eta}_{2^r})^*:[\Sigma^{3+k}M(\z/{2^r}),Top/O]\to\pi_{5+k}(Top/O)$ is 
			\begin{itemize}
				\item[(i)] trivial for $k=1,2,3,7,8,9.$
				\item[(ii)] $\z/2$ for $k=5,10.$
			\end{itemize}
			\item [(b)] The image of $(\tilde{\eta}_{2^r})^*:[\Sigma^7 M(\z/{2^r}),Top/O]\to \pi_{9}(Top/O)$ equals
				\begin{itemize}
					\item [(i)] $\z/2\oplus\z/2$ if $r=1$ and $2.$
					\item [(ii)] $\z/2$ if $r\geq 3.$
			\end{itemize}
			\item[(c)] The image of $(\tilde{\eta}_{2^r})^*:[\Sigma^9 M(\z/{2^r}),Top/O]\to \pi_{11}(Top/O)$ equals 
			\begin{itemize}
				\item[(i)] $\z/4$ if $r=1.$
				\item [(ii)] $\z/2$ if $r\geq 2.$
			\end{itemize}
		\end{itemize}
	\end{lemma}

Before proving this lemma, we first prove a preliminary result that will be used in the argument.
\begin{lemma}\label{etatildeG/O}
   The map $(\tilde{\eta}_{2^r})^*: [\Sigma^7 M(\mathbb{Z}/{2^r}), G/O]\to \pi_9(G/O)$ is surjective.
\end{lemma}
\begin{proof}
We observe from the long exact sequence induced from \eqref{Moore} along $G/O$ that $$[\Sigma^7 M(\mathbb{Z}/{2^r}), G/O]\cong \mathbb{Z}/{2^r}\oplus\mathbb{Z}/2.$$ Since $q^{2^r}_8\circ \tilde{\eta}_{2^r}=\eta,$ we obtain the following commutative diagram:  
\[\begin{tikzcd}
	{\pi_8(G/O)} && {[\Sigma^7 M(\mathbb{Z}/{2^r}), G/O]\cong \mathbb{Z}/{2^r}\oplus\mathbb{Z}/2} & 0 \\
	{\pi_9(G/O)} \\
	{0.}
	\arrow["{(q^{2^r}_8)^*}", from=1-1, to=1-3]
	\arrow["{\eta_8^*}"', from=1-1, to=2-1]
	\arrow[from=1-3, to=1-4]
	\arrow["{(\tilde{\eta}_{2^r})^*}", from=1-3, to=2-1]
	\arrow[from=2-1, to=3-1]
\end{tikzcd}\]
Here, the surjectivity of \( (q^{2^r}_8)^* : \pi_8(G/O) \to [\Sigma^7 M(\mathbb{Z}/{2^r}), G/O] \) follows from the long exact sequence induced from \eqref{Moore} along \( G/O \), together with \( \pi_7(G/O) = 0 \), and \( \eta_8^* : \pi_8(G/O) \to \pi_9(G/O) \) is surjective by Lemma~\ref{etaG/O}. The commutativity of the diagram then implies that the map   
\( (\tilde{\eta}_{2^r})^* : [\Sigma^7 M(\mathbb{Z}/{2^r}), G/O] \to \pi_9(G/O) \) is surjective as well.    
\end{proof}
    
	\begin{proof}[Proof of Lemma~\ref{etatilde}]
 We examine each value of $k$ one by one to prove the result. By the definition of $\tilde{\eta}_{2^r},$ we have the following commutative diagram:
       \begin{equation}\label{tildeetadiagram}
           \begin{tikzcd}
               {\pi_{4+k}(Top/O)}\arrow[rr,"(q^{2^r}_{4+k})^*"]\arrow[dr,"\eta^*"']&&{[\Sigma^{3+k}M(\mathbb{Z}/{2^r}),Top/O]}\arrow[dl,"(\Tilde{\eta}_{2^r})^*"]\\
               & {\pi_{5+k}(Top/O).}
           \end{tikzcd}
       \end{equation}
 
  For $k=1,$ and $7$, the result is immediate due to the triviality of $\pi_{5+k}(Top/O).$ Additionally, from the exact sequence \eqref{longMoore}, it can be observed that the groups $[\Sigma^{5} M(\mathbb{Z}/{2^r}), Top/O]$ and $[\Sigma^{11}M(\mathbb{Z}/{2^r}),Top/O]_{(3)}$ are both trivial. Consequently, the map $$(\tilde{\eta}_{2^r})^*:[\Sigma^{3+k}M(\z/{2^r}),Top/O]\to \pi_{5+k}(Top/O)$$ is trivial for $k=2$ and $8.$ This concludes the cases $k=2$ and $8.$ 
       
For both cases $k=3$ and $9$, the map $(\eta)^*:\pi_{4+k}(Top/O)\to\pi_{5+k}(Top/O)$ is trivial as per Lemma \ref{lemma2.17}. Moreover, the map $(q^{2^r}_{4+k})^*:\pi_{4+k}(Top/O)\to[\Sigma^{3+k}M(\mathbb{Z}/2^r),Top/O]$ is surjective, as seen from the exact sequence \eqref{longMoore}. The combination of these observations with the commutativity of the diagram \eqref{tildeetadiagram} implies that the result holds for these cases.
  
 \textbf{Consider the cases $k=5$ and $10$ :} In these particular cases, Lemma \ref{lemma2.17} implies that the image of the map $\eta^*:\pi_{4+k}(Top/O)\to\pi_{5+k}(Top/O)$ is $\mathbb{Z}/2$. From the sequence \eqref{longMoore}, we can deduce the following :
 \begin{itemize}
     \item The map $(q^{2^r}_{9})^*:\pi_9(Top/O)\to[\Sigma^{8} M(\mathbb{Z}/{2^r}),Top/O]$ is injective,
     \item $[\Sigma^8 M(\mathbb{Z}/{2^r}),Top/O]_{(3)}$ is trivial,
     \item The map $(q^{2^r}_{14})^*: \pi_{14}(Top/O)\to [\Sigma^{13}M(\mathbb{Z}/{2^r}),Top/O]$ is an isomorphism.
 \end{itemize}
By combining these observations in the commutative diagram \eqref{tildeetadiagram}, we can obtain the desired result.

 \textbf{The case $k=4$:} Using the long exact sequence induced from \eqref{Moore} along $G/Top$ and $G/O$, we get $[\Sigma^8 M(\mathbb{Z}/{2^{r}}), G/Top]=0$ and $[\Sigma^7 M(\mathbb{Z}/{2^r}), G/O]=\mathbb{Z}/{2^r}\oplus\mathbb{Z}/2$, respectively.

Since $[\Sigma^7 M(\mathbb{Z}/{2^r}), \Omega(G/Top)]=0$, the long exact sequence 
\[\cdots\to [\Sigma^7M(\mathbb{Z}/{2^r}), \Omega (G/Top)]\xrightarrow{\omega_*}[\Sigma^7 M(\mathbb{Z}/{2^r}), Top/O]\xrightarrow{\psi_*}[\Sigma^7M(\mathbb{Z}/{2^r}), G/O],\]
induced from the fiber sequence $\Omega(G/Top)\xrightarrow{\omega}Top/O\xrightarrow{\psi}G/O\xrightarrow{\phi}G/Top$ implies that the map $$\psi_*:[\Sigma^7 M(\mathbb{Z}/{2^r}), Top/O]\to [\Sigma^7 M(\mathbb{Z}/{2^r}), G/O]$$ is injective. 

We note from \eqref{longMoore} that for \( r \geq 2 \), the group \( [\Sigma^7 M(\mathbb{Z}/2^r), Top/O] \) is either \( \mathbb{Z}/2 \oplus \mathbb{Z}/4 \) or \( \mathbb{Z}/8 \), and for \( r = 1 \), it is either \( \mathbb{Z}/2 \oplus \mathbb{Z}/2 \) or \( \mathbb{Z}/4 \). Now, using the injectivity of the map  
\[
\psi_*: [\Sigma^7 M(\mathbb{Z}/2^r), Top/O] \longrightarrow [\Sigma^7 M(\mathbb{Z}/2^r), G/O],
\] 
together with the fact that \( [\Sigma^7 M(\mathbb{Z}/2^r), G/O] \cong \mathbb{Z}/2^r \oplus \mathbb{Z}/2 \), we conclude that \( \psi_* \) is an isomorphism for \( r = 1 \) and \( r = 2 \). Therefore,
\[
[\Sigma^7 M(\mathbb{Z}/2), Top/O] \cong \mathbb{Z}/2 \oplus \mathbb{Z}/2 \quad \text{and} \quad [\Sigma^7 M(\mathbb{Z}/4), Top/O] \cong \mathbb{Z}/2 \oplus \mathbb{Z}/4.
\]
For the case $r=1$ and $2$, we consider the following commutative diagram:
\[\begin{tikzcd}
	{[\Sigma^7 M(\mathbb{Z}/{2^r}), Top/O]} && {\pi_9(Top/O)} \\
	{[\Sigma^7 M(\mathbb{Z}/{2^r}), G/O]} && {\pi_9(G/O)} \\
	&& {0.}
	\arrow["{(\tilde{\eta}_{2^r})^*}", from=1-1, to=1-3]
	\arrow["{\psi_*}"', "\cong",from=1-1, to=2-1]
	\arrow["{\psi_*}", from=1-3, to=2-3]
	\arrow["{(\tilde{\eta}_{2^r})^*}"', from=2-1, to=2-3]
	\arrow[from=2-3, to=3-3]
\end{tikzcd}\]
Since both the maps $\psi_*:\pi_9(Top/O)\to \pi_9(G/O)$ and  \( (\tilde{\eta}_{2^r})^* : [\Sigma^7 M(\mathbb{Z}/{2^r}), G/O] \to \pi_9(G/O) \) are surjective by Lemma~\ref{etatildeG/O}, the above commutative diagram implies that the image of \( (\tilde{\eta}_{2^r})^* : [\Sigma^7 M(\mathbb{Z}/{2^r}), Top/O] \to \pi_9(Top/O) \) is \( \mathbb{Z}/2 \oplus \mathbb{Z}/2 \) for $r=1$ and $2$. This completes the proof of part {(b)}{(i)}.

Let's now assume that $r\geq 3$. Then the commutative diagram \eqref{diagrammoore} induces the following diagram of long exact sequences along $Top/O$.  
\begin{center}
	\begin{tikzcd}[column sep=1.2 em]
		\arrow[r,"0"]&{\pi_8(Top/O)}\arrow[r,"(q^{2^r}_8)^*"]\arrow[d,-,double equal sign distance,double]&{[\Sigma^7 M(\z/{2^r}),Top/O]}\arrow[d,"(\lambda_{1,r})^*"']\arrow[r,"(i^{2^r}_7)^*"]&{\pi_7(Top/O)}\arrow[d,"\times 2^{r-1}"]\arrow[r,"\times 2^r"]&{\pi_7(Top/O)}\arrow[d,-,double equal sign distance,double]\\
		\arrow[r,"0"']&{\pi_8(Top/O)}\arrow[r,"(q^{2}_8)^*"']&{[\Sigma^7 M(\z/2),Top/O]\cong \z/2\oplus\z/2}\arrow[r,"(i^{2}_7)^*"']&{\pi_7(Top/O)}\arrow[r,"\times 2"']&{\pi_7(Top/O).}
	\end{tikzcd}
\end{center}
From this, we deduce that the image of $$(\lambda_{1,r})^*: [\Sigma^7 M(\mathbb{Z}/{2^r}), Top/O]\to [\Sigma^7 M(\mathbb{Z}/2), Top/O]$$ is $\mathbb{Z}/2$ and $[\Sigma^7 M(\mathbb{Z}/{2^r}), Top/O]\cong \mathbb{Z}/2\oplus\mathbb{Z}/4$ for $r\geq 3$.

Combining these observations in the following commutative diagram given by the  relation \eqref{Bchirelation}
\begin{center}
	\begin{tikzcd}
		{[\Sigma^7 M(\mathbb{Z}/{2^r}),Top/O]}\arrow[rr,"(\lambda_{1,r})^*"]\arrow[dr,"(\Tilde{\eta}_{2^r})^*"']&& {[\Sigma^7 M(\mathbb{Z}/2),Top/O]\cong \mathbb{Z}/2\oplus\mathbb{Z}/2}\arrow[dl,"(\Tilde{\eta}_{2})^*"]\\
		& {\pi_9(Top/O)}
	\end{tikzcd}
\end{center}
establishes the statement {(b)}{(ii)}. 

\textbf{The case $k=6$:} In this case, it follows from the long exact sequence \eqref{longMoore} for $r=1$, Lemma \ref{lemma2.17} and \cite[Lemma 2.3]{DCTSWS2018} that  there is a non-split sequence 
\begin{equation*}
0\to \mathbb{Z}/2\left(\subset \pi_{10}(Top/O)\right)\to [\Sigma^9 M(\z/{2}), Top/O]=\z/4\oplus\z/2\oplus\z/2 \to \pi_{9}(Top/O)\to 0, 
\end{equation*}
where the map $(q^2_{10})^*: \pi_{10}(Top/O)\to [\Sigma^9 M(\mathbb{Z}/2),Top/O]$ sends the generator $\mu \circ\eta$ to the element of order $2$ in the summand $\z/4$. Further, from Lemma \ref{lemma2.17}, the image of $\eta^*:\pi_{10}(Top/O)\to \pi_{11}(Top/O)$  is $\mathbb{Z}/2$. Combining these facts in the diagram \eqref{tildeetadiagram} implies that the image of $(\tilde{\eta}_2)^*:[\Sigma^9 M(\z/2),Top/O]\to \pi_{11}(Top/O)$ is $\z/4.$ This completes the proof of the statement {(c)}{(i)}. Assume $r\geq 2$ and consider the following diagram 
\begin{center}
	\begin{tikzcd}
		{0}\arrow[r]&{\mathbb{Z}/2}\arrow[r]\arrow[d,-,double equal sign distance,double]&{[\Sigma^9 M(\z/{2^r}),Top/O]}\arrow[d,"(\lambda_{1,r})^*"]\arrow[r,"(i_9)^*"]&{\pi_9(Top/O)}\arrow[d,"0"']\arrow[r,"0"]&{0}\\
		{0}\arrow[r]&{\mathbb{Z}/2}\arrow[r]&{[\Sigma^9 M(\z/2),Top/O]\cong \z/4\oplus\z/2\oplus\z/2}\arrow[r,"(i_9)^*"']&{\pi_9(Top/O)}\arrow[r,"0"']&{0,}
	\end{tikzcd}
\end{center}
where the first and second rows are induced by the cofiber sequences of the mapping cones $\Sigma^9 M(\z/2^r)$ and $\Sigma^9 M(\z/2)$ respectively as in \eqref{diagrammoore}. From this diagram, we can deduce that $[\Sigma^9 M(\z/{2^r}), Top/O]\cong   \bigoplus\limits_{i=1}^4 \mathbb{Z}/2$. Now apply the same argument as in the case $r=1$ using the diagram \eqref{tildeetadiagram},  we get that the image of $(\tilde{\eta}_{2^r})^*:[\Sigma^9 M(\z/2^r),Top/O]\to \pi_{11}(Top/O)$ is $\z/2.$ This proves the statement {(c)}{(ii)}.

This finishes the proof of the lemma.
	\end{proof}
The following theorem determines the concordance inertia group of $M\times\mathbb{S}^k$ for some $k\geq 1$.
\begin{theorem}\label{inerM^5}
	Let $M$ be a simply connected, closed, oriented, smooth $5$-manifold with homology of the form (\ref{homologyM^5}). The following statements hold.
	\begin{enumerate}
	
		\item[(1)] If $M$ is a spin manifold, then $I_c(M\times\s^k)=0$ for $k\geq1.$ 
			
		\item[(2)] Suppose $M$ is a non-spin manifold.
                  \begin{enumerate}
                      \item $I_c(M\times\mathbb{S}^k)=0$ for $k=2,3,7,8,$ or $9.$
                      \item For $k=4,5,6,$ or $10,$
                      \begin{enumerate}
                      \item if $M$ satisfies \eqref{5non-spin1}, then 
                       $I_c(M\times\mathbb{S}^k)=\mathbb{Z}/2.$
                      \item if $M$ satisfies \eqref{5non-spin2}, then 
                      \begin{enumerate}
                          \item  	\[ I_c(M\times\s^4)= \begin{cases}  \mathbb{Z}/2\oplus\mathbb{Z}/2, &\mbox{ if $r=1$ or $2$},\\ \mathbb{Z}/2, &\mbox{if $r\geq 3.$}  \end{cases} \] 
				\item $I_c(M\times\s^5)=\mathbb{Z}/2.$
				\item  \[I_c(M\times\s^6)=\begin{cases} 
					 \mathbb{Z}/4, & \mbox{if $r=1$},\\
					 \mathbb{Z}/2, &\mbox{if $r\geq 2.$}
				\end{cases}\]
                    \item $I_c(M\times\mathbb{S}^{10})=\mathbb{Z}/2.$
                    \end{enumerate}
                      \end{enumerate}
                  \end{enumerate}
			Here $r$ is the least positive integer such that $\mathbb{Z}/{2^r}$ is in the torsion part of $H_2(M;\mathbb{Z}).$
		\end{enumerate}
	\end{theorem}
\begin{proof} 
	The results {(1)} and  {(2)} are derived from Lemma~\ref{concordanceinertia} by using Proposition~\ref{5manspin} and the stable splitting \eqref{5non-spin1}, \eqref{5non-spin2}, along with the calculations from Lemmas~\ref{lemma2.17} and~\ref{etatilde}.
\end{proof}

As a direct consequence of the stable homotopy type of $M$ described in Proposition \ref{5manspin} and \eqref{5non-spin1}, \eqref{5non-spin2}, we have the following:

\begin{prop}\label{cM^5}
	Let $M$ be a simply-connected, closed, oriented, smooth $5$-dimensional manifold with homology of the form (\ref{homologyM^5}). The following holds:
	\begin{itemize}
		\item [(a)] If $M$ is a spin manifold, then $[\Sigma^k M,Top/O]$ is isomorphic to $$\pi_{5+k}(Top/O)\oplus \\\bigoplus\limits_{i=1}^d\left(\pi_{2+k}(Top/O)\oplus \pi_{3+k}(Top/O)\right)\oplus  
   \bigoplus\limits_{j=1}^t[\Sigma^{2+k} M(\z/p_j^{r_j}),Top/O].$$

   \item[(b)] If $M$ is a non-spin manifold, then $[\Sigma^k M, Top/O]$ is isomorphic to either
\begin{align*}
&[\Sigma^{k+1}\mathbb{C}P^2,Top/O]\oplus 
\bigoplus_{i=1}^{d-1}\left(\pi_{2+k}(Top/O)\oplus \pi_{3+k}(Top/O)\right)\\
&\qquad\oplus \bigoplus_{j=1}^t[\Sigma^{2+k}M(\mathbb{Z}/{p_j}^{r_j}), Top/O] 
\oplus \pi_{2+k}(Top/O)
\end{align*}
\begin{center}
    or
\end{center}
\begin{align*}
&[\Sigma^{k} Cone(\tilde{\eta}_{2^r}),Top/O]\oplus 
\bigoplus_{i=1}^{d}\left(\pi_{2+k}(Top/O)\oplus\pi_{3+k}(Top/O)\right)\\
&\qquad\oplus\bigoplus_{j=2}^t[\Sigma^{2+k}M(\mathbb{Z}/{p_j}^{r_j}),Top/O].
\end{align*}
	\end{itemize}
\end{prop}

In conclusion, according to Corollary \ref{concorMtimesSk}, if $M$ is either a 4-dimensional manifold or a simply connected 5-dimensional manifold, then the smooth concordance structure set $\mathcal{C}(M\times \mathbb{S}^k)$ can be determined using Proposition \ref{cM^4} and Proposition \ref{cM^5}.

	
\section{Smooth structures on \texorpdfstring{$\mathbb{C}P^2\times\mathbb{S}^k$}{}}\label{application}
In this section, we use the computations from the previous sections to provide a diffeomorphism classification of smooth structures on $\mathbb{C}P^2\times\mathbb{S}^k,$ where $4\leq k\leq 6.$ To do this, we recall the surgery exact sequence and some related concepts from surgery theory \cite{IMRLBW80, ranicki, DC10, DCIH15}.
\begin{defn}[Relative Smooth Structure Set]
    Consider a compact smooth manifold $M$ of dimension $n$ with possibly empty boundary $\partial M.$ The \textit{relative smooth structure set} of $M,$
    $$\mathcal{S}^{Diff}(M ~\rel ~\partial M):= \{(N,f): f: (N, \partial N) \to (M,\partial M)\}/{\simeq}$$
consists of the equivalence classes of pairs $(N,f),$ where $N$ is a compact smooth manifold with boundary $\partial N$ and $f: (N, \partial N) \to (M,\partial M)$ is a homotopy equivalence such that the restircted map $f\vert_{\partial N}: \partial N \to \partial M$ is a diffeomorphism. Two pairs $(N_1,f_1)$ and $(N_2,f_2)$ are equivalent if there exists a diffeomorphism $H: (N_1, \partial N_1) \to (N_2,\partial N_2) $ such that $f_2\circ H$ is homotopic to $f_1$ rel $\partial N_1.$ The equivalence class of $(N,f)$ is denoted by $[(N,f)].$ The base point of $\mathcal{S}^{Diff}(M \rel \partial M)$ is the equivalence class of $(M, Id).$

If $\partial M$ is empty, we abbreviate $\mathcal{S}^{Diff}(M):=\mathcal{S}^{Diff}(M ~\rel ~\varnothing).$ 
\end{defn} 
 
The main tool for calculating the smooth structure set $\mathcal{S}^{Diff}(M ~\rel~ \partial M)$ of an $n$-dimensional manifold $M$ with $n\geq 5$ is the surgery exact sequence
\begin{equation}\label{relativesmoothsugeryexact}
    \cdots\to L_{n+1}(\pi_1(M))\xrightarrow{\omega_{\rel}^{Diff}} \mathcal{S}^{Diff}(M ~\rel ~\partial M)\xrightarrow{\eta_{\rel}^{Diff}} [M/{\partial M},G/O]\xrightarrow{\Theta_{\rel}^{Diff}} L_n(\pi_1(M)),
\end{equation}
 
Observe that there is a forgetful map $\mathcal{F}: \mathcal{C}(M)\to \mathcal{S}^{Diff}(M)$ such that the following diagram commutes:
\begin{center}
    \begin{tikzcd}
        {\mathcal{S}^{Diff}(M)}\arrow[r,"\eta^{Diff}"]&{[M,G/O]}\\
        {\mathcal{C}(M).}\arrow[u,"\mathcal{F}"]\arrow[ur,"\psi_*"']
    \end{tikzcd}
\end{center}

Recall some facts from \cite{IJ63, JBRS74, RS87, bel} related to smooth and tangential structure sets, which will be used later in this section.

Consider $E_1(\mathbb{C}P^q)$ to be the path-component of the space of all self maps of $\mathbb{C}P^q$ that include $Id_{\mathbb{C}P^q},$ and let $F_{\mathbb{S}^1}(\C^{q+1})$ represent the space of all $\mathbb{S}^1$-equivariant self maps of $\mathbb{S}^{2q+1}.$ Additionally, there is a stabilization map $s_{q+1}: F_{\mathbb{S}^1}(\C^{q+1})\to F_{\mathbb{S}^1}(\C^{q+2}),$ and we define $F_{\mathbb{S}^1}$ as the limit of these stabilizations.  

\begin{fact}\noindent\label{FS1}
\begin{itemize}
      
    \item[(a)] There is a group homomorphism $\alpha_{\mathbb{S}^k}: \pi_k(E_1(\mathbb{C}P^{q}))\to \mathcal{E}(\mathbb{C}P^q\times\mathbb{S}^k)$ for all $k,q \geq 1,$ where $\mathcal{E}(X)$ denotes the group of all homotopy classes of based self-homotopy equivalences of $X$ \cite[Page 15]{bel}.
    
    \item [(b)] For $q\geq 1$ and $k\geq 2,$ there is a group homomorphism $\Psi: \pi_k(E_1(\mathbb{C}P^q)) \to \mathcal{S}^{Diff}(\mathbb{C}P^q\times \mathbb{D}^k ~\rel ~\mathbb{C}P^q\times\mathbb{S}^{k-1}),$ and a canonical base-point preserving map $\Gamma: \mathcal{S}^{Diff}(\mathbb{C}P^q\times\mathbb{D}^k~\rel~\mathbb{C}P^q\times\mathbb{S}^{k-1}) \to \mathcal{S}^{Diff}(\mathbb{C}P^q\times\mathbb{S}^k)$ \cite[Page 24]{bel}. Moreover, the composition $\Gamma\circ\Psi$ sends the element $\alpha \in \pi_k(E_1(\mathbb{C}P^q))$ to the element $[(\mathbb{C}P^q\times\mathbb{S}^k,f)]$ in $\mathcal{S}^{Diff}(\mathbb{C}P^q\times\mathbb{S}^k),$ where $f$ represents a tangential self-homotopy equivalence of $\mathbb{C}P^q\times\mathbb{S}^k$ that comes from $\alpha$ \cite[Proposition 7.3]{bel}.

    \item[(c)] The homotopy groups $\pi_k(F_{\mathbb{S}^1}(\mathbb{C}^{q+1}))$ and $\pi_k(E_1(\mathbb{C}P^q))$ are isomorphic for all $q\geq 1$ and $k\geq 2$ (see \cite[Theorems 2.1-2.2]{IJ63} and \cite[Theorem 11.1]{JBRS74}). 
   
    \item[(d)] For all $q\geq 1,$ $\pi_k(F_{\mathbb{S}^1}(\mathbb{C}^{q+1}))$ is isomorphic to $\pi_k(F_{\mathbb{S}^1})$ if $1\leq k\leq 2q$ \cite[Proposition 6.1]{bel}.
    \item[(e)] $\pi_k(F_{\mathbb{S}^1})\cong \pi_{k}^s(\Sigma\mathbb{C}P^{\infty})\oplus\pi_{k-1}^s$ for $k\geq 1$ (see \cite[Theorem 6.6]{JBRS74} and \cite[Page 18]{bel}).
    \item[(f)] Let $\alpha \in \pi_{k-1}^s$ correspond to the element $f\in \pi_k(F_{\mathbb{S}^1}).$ Then the normal invariant of $f$ is the image of the composition $\mathbb{S}^k \wedge \mathbb{C}P_{+}^{\infty}\xrightarrow{\alpha \wedge Id_{\mathbb{C}P_{+}^{\infty}}} \mathbb{S}^1\wedge \mathbb{C}P_{+}^{\infty}\xrightarrow{t}\mathbb{S}^0$ under the canonical map $j_*:[\Sigma^k \mathbb{C}P_{+}^{\infty},G]\to [\Sigma^k \mathbb{C}P_{+}^{\infty},G/O]$, \cite[Page 193]{RS87} where $t: \mathbb{S}^1\wedge\mathbb{C}P_{+}^{\infty}\to \mathbb{S}^0$ is the Umkehr or tranfer map \cite{JBRS74}. Further, we note that $t\circ (\alpha\wedge Id_{\mathbb{C}P_{+}^{\infty}})$ is stably homotopic to $ \alpha \circ \Sigma^{k-1} t$ for $k\geq 2$. 
\end{itemize}
\end{fact}
\begin{note}\label{non_triviality_of_Umkehr_map}
We note from \cite[Page 33]{bel} that the restriction of the Umkehr map $t$ to $\Sigma \mathbb{C}P^1$ is a generator of $\pi_3^s\cong \mathbb{Z}/{24}$. Using this fact, it follows from the following diagram that
\[\begin{tikzcd}
	t &&& {t\vert_{\Sigma \mathbb{C}P^1}} \\
	{\{\Sigma\mathbb{C}P^{\infty},\mathbb{S}^0\}} &&& {\{\mathbb{S}^3,\mathbb{S}^0\}\cong \pi_3^s} \\
	{\{\Sigma^4\mathbb{C}P^{\infty},\mathbb{S}^3\}} &&& {\{\mathbb{S}^6,\mathbb{S}^3\}\cong \pi_3^s} \\
	{\Sigma^3 t} &&& {\Sigma^3 t\vert_{\Sigma^4\mathbb{C}P^1}}
	\arrow[maps to, from=1-1, to=1-4]
	\arrow[bend right=60 , maps to, from=1-1, to=4-1]
	\arrow[bend left=60, maps to, from=1-4, to=4-4]
	\arrow[from=2-1, to=2-4]
	\arrow["{\Sigma^3}"', from=2-1, to=3-1]
	\arrow["{\Sigma^3}", "\cong"',from=2-4, to=3-4]
	\arrow[from=3-1, to=3-4]
	\arrow[maps to, from=4-1, to=4-4]
\end{tikzcd}\]
$\Sigma^{3} t : \Sigma^{4} \mathbb{C}P_{+}^{\infty}\to \mathbb{S}^3$ is non-trivial.
\end{note}

Let’s turn our attention to the product manifold $\mathbb{C}P^2\times\mathbb{S}^4$. By applying the previously mentioned facts in conjunction with surgery theory, we prove the following.

\begin{lemma}\label{epsilon}
  There exists a tangential self-homotopy equivalence $f_{2\nu}:\mathbb{C}P^2\times\mathbb{S}^4\to \mathbb{C}P^2\times\mathbb{S}^4,$ for which the normal invariant is non-zero. In particular, the normal invariant of $f_{2\nu}$ equals $f_{\mathbb{C}P^2\times\mathbb{S}^4}^*([\epsilon]),$ where $[\epsilon]$ represents the generator of $2$-torsion of $\pi_8(G/O)\subseteq [\mathbb{C}P^2\times\mathbb{S}^4,G/O]$.
\end{lemma}
\begin{proof}
By Fact~\ref{FS1}{(d)} and {(e)}, we get $\pi_4(F_{\mathbb{S}^1}(\mathbb{C}^3))\cong \pi_4(F_{\mathbb{S}^1})\cong \mathbb{Z}/{8}\{\nu\}\oplus\mathbb{Z}/3 \{\alpha_1\}.$ Let $f_{2\nu}:\mathbb{C}P^2\times\mathbb{S}^4\to \mathbb{C}P^2\times\mathbb{S}^4$ be a self-homotopy equivalence that is induced by the element $2\nu \in \pi_{4}(F_{\mathbb{S}^1}(\mathbb{C}^3)).$ Now, it follows from Fact~\ref{FS1}{(f)} together with $\pi_4^s=0$ that the normal invariant $\eta^{Diff}(f_{2\nu})$ is the image of the composition map $(2\nu)\circ \Sigma^3 t:\Sigma^4\mathbb{C}P^2\to\mathbb{S}^0$ in $[\Sigma^4 \mathbb{C}P^2,G/O]\subset [\mathbb{C}P^2\times\mathbb{S}^4,G/O]$, where by Note \ref{non_triviality_of_Umkehr_map} the restriction of $\Sigma^3 t$ to $\Sigma^4 \mathbb{C}P^1$ is a generator of $\pi_3^s$. Since $2\nu^2=0$ \cite[Page 189]{toda}, it follows that $2\nu \circ \Sigma^3 t\vert_{\mathbb{S}^6}=0$. This implies that there exists a map $\mathcal{T}:\mathbb{S}^8\to \mathbb{S}^0$ such that the following diagram commutes: 
\[
\begin{tikzcd}
	{\mathbb{S}^7} & {\mathbb{S}^6} & {\Sigma^4\mathbb{C}P^2} & {\mathbb{S}^8} & {\mathbb{S}^7} \\
	&& {\mathbb{S}^3} \\
	&& {\mathbb{S}^0.}
	\arrow["\eta", from=1-1, to=1-2]
	\arrow["{\Sigma^4 i}", from=1-2, to=1-3]
	\arrow["\nu"', from=1-2, to=2-3]
	\arrow["{\Sigma^4 f_{\mathbb{C}P^2}}", from=1-3, to=1-4]
	\arrow["{\Sigma^3 t}", from=1-3, to=2-3]
	\arrow["\eta", from=1-4, to=1-5]
	\arrow["\mathcal{T}", dashed, from=1-4, to=3-3]
	\arrow["{2\nu}"', from=2-3, to=3-3]
\end{tikzcd}\]
Note that we may view the map $\mathcal{T}:\mathbb{S}^8\to \mathbb{S}^0$ as an element of the Toda bracket $<\eta, \nu, 2\nu>$. By \cite[Page-189]{toda}, $<\eta, \nu, 2 \nu>=\{\Bar{\nu}, \epsilon\}$ which projects non-trivially to the class $[\epsilon]$ in $\pi_8(G/O).$ This implies that $\eta^{Diff}(f_{2\nu})$ is the image of the class $[\epsilon]$ under the map $(\Sigma^4 f_{\mathbb{C}P^2})^*:\pi_8(G/O)\to [\Sigma^4\mathbb{C}P^2,G/O].$ Now the result follows from the following commutative diagram:

\begin{center}
    \begin{tikzcd}
         &{[\Sigma^4 \mathbb{C}P^2,G/O]}\arrow[dr,"p^*"]\\
            {\pi_8(G/O)}\arrow[ur,"(\Sigma^4 f_{\mathbb{C}P^2})^*"]\arrow[rr,"(f_{\mathbb{C}P^2\times\s^4})^*"']&&{[\mathbb{C}P^2\times\mathbb{S}^4,G/O],}
    \end{tikzcd}
\end{center}
where the map $p^*:[\Sigma^4 \mathbb{C}P^2,G/O]\to [\mathbb{C}P^2\times\mathbb{S}^4,G/O]$ is injective.
\end{proof}

The next theorem identifies the smooth structure on $\mathbb{C}P^{2}\times \mathbb{S}^{4}$.
\begin{thm}\label{classificationCP2S4}
    Let $N$ be a closed, oriented, smooth manifold homeomorphic to $\mathbb{C}P^2\times\mathbb{S}^4$. Then $N$ is oriented diffeomorphic to $\mathbb{C}P^2\times\mathbb{S}^4.$ 
\end{thm}

  \begin{proof}
  Suppose we have a homeomorphism $g: N\to \mathbb{C}P^2\times\mathbb{S}^4$. This yields an element $[(N,g)] \in \mathcal{C}(\mathbb{C}P^2\times\s^4).$ 
    Consider the following commutative diagram
    \begin{center}
        \begin{tikzcd}
             {0}\arrow[d] &&{0}\arrow[d]\\
            {\Theta_8}\arrow[rr,"(f_{\mathbb{C}P^2\times\s^4})^*","\cong"']\arrow[d,"\psi_*"']&&{\mathcal{C}(\mathbb{C}P^2\times\mathbb{S}^4)\cong [\Sigma^4\mathbb{C}P^2,Top/O]}\arrow[d,"\psi_*"]\\
            {\pi_8(G/O)}\arrow[dr,"(\Sigma^4 f_{\mathbb{C}P^2})^*"']\arrow[rr,"(f_{\mathbb{C}P^2\times\s^4})^*"'] &&{[\mathbb{C}P^2\times\s^4,G/O]}  \\
            &{[\Sigma^4 \mathbb{C}P^2,G/O],}\arrow[ur,"p^*"']
            \end{tikzcd}
    \end{center}
    where the induced degree one map $f^{*}_{\mathbb{C}P^2\times\s^4}:  \Theta_8\to \mathcal{C}(\mathbb{C}P^2\times\mathbb{S}^4)$ is an isomorphism by Corollary \ref{concorMtimesSk} and Proposition \ref{prop3.6}{(3)}, and the map $(\Sigma^4 f_{\mathbb{C}P^2})^*:\pi_8(G/O)\to [\Sigma^4\mathbb{C}P^2,G/O]$ is injective as $\pi_7(G/O)=0$. From this diagram, it’s clear that the normal invariant $\eta^{Diff}(g)$ is either trivial or equal to $(f_{\mathbb{C}P^2\times\mathbb{S}^4})^*([\epsilon]).$ 
    If $\eta^{Diff}(g)$ equals $(f_{\mathbb{C}P^2\times\mathbb{S}^4})^*([\epsilon])$ then by applying Lemma \ref{epsilon} and the composition formula for normal invariants, we find that  
    \begin{align*}
        \eta^{Diff}(f_{2\nu}\circ g)= & \eta^{Diff}(f_{2\nu})+ (f_{2\nu}^{-1})^* \eta^{Diff}(g)\\
         = & (f_{\mathbb{C}P^2\times\mathbb{S}^4})^*([\epsilon])\pm (f_{\mathbb{C}P^2\times\mathbb{S}^4})^*([\epsilon])=0.
    \end{align*}
    This implies that $(N,f_{2\nu}\circ g)$ and $(\mathbb{C}P^2\times\s^4,Id)$ represent the same element in $\mathcal{S}^{Diff}(\mathbb{C}P^2\times\s^4)$ from the surgery exact sequence \eqref{relativesmoothsugeryexact}. This concludes the proof of the theorem.
  \end{proof}
	
Next, let’s focus on the normal invariants associated with the smooth structures of $\mathbb{C}P^2\times\mathbb{S}^5$. We initiate our discussion by proving the following:

\begin{lemma}\label{psicp2s5}
    The map $\psi_*:[\mathbb{C}P^2\times\mathbb{S}^5,Top/O]\to [\mathbb{C}P^2\times\mathbb{S}^5,G/O]$ is trivial.
\end{lemma}
\begin{proof}
    Since the short exact sequence \eqref{note2.3} splits along both $Top/O$ and $G/O,$ there is a commutative diagram where each of the three components of $[\mathbb{C}P^2\times\mathbb{S}^5,Top/O]$ is mapped into the corresponding component of $[\mathbb{C}P^2\times\mathbb{S}^5,G/O]$ under the induced homomorphism by the canonical map $\psi:Top/O \to G/O.$ From this diagram, by  using the facts that both $[\Sigma^5\mathbb{C}P^2,G/O]$ and $[\mathbb{C}P^2,Top/O]$ are zero, we find that the map $\psi_*:[\mathbb{C}P^2\times\mathbb{S}^5,Top/O]\to [\mathbb{C}P^2\times\mathbb{S}^5,G/O]$ is trivial, as desired.
\end{proof}	

\begin{thm}\label{classificationCP2S5}
    If a closed smooth oriented manifold $N$ is homeomorphic to $\mathbb{C}P^2\times\mathbb{S}^5,$ then it is oriented diffeomorphic to $\mathbb{C}P^2\times\mathbb{S}^5.$
\end{thm}
\begin{proof}
Let $g:N\to \mathbb{C}P^2\times \mathbb{S}^5$ be a homeomorphism. We show that the pairs $(N,g)$ and $(\mathbb{C}P^2\times\mathbb{S}^5,Id)$ represent the same element in $\mathcal{S}^{Diff}(\mathbb{C}P^2\times\mathbb{S}^5).$  By applying Lemma \ref{psicp2s5}, we find that the normal invariant $\eta^{Diff}(g)=0.$ Now it follows from the smooth surgery exact sequence (\ref{relativesmoothsugeryexact}) for $\mathbb{C}P^2\times\mathbb{S}^5$ that there exists $\Sigma \in bP_{10}=\mathbb{Z}/2$ such that $(N,g)$ and $( \mathbb{C}P^2\times\mathbb{S}^5\# \Sigma, h)$ are equivalent in $\mathcal{S}^{Diff}(\mathbb{C}P^2\times\mathbb{S}^5),$ where $h: (\mathbb{C}P^2\times\mathbb{S}^5)\# \Sigma\to \mathbb{C}P^2\times\mathbb{S}^5$ is the canonical homeomorphism. We claim that $\Sigma$ is (oriented) diffeomorphic to the standard $9$-sphere, which is equivalent to prove from the surgery exact sequence that $\Theta^{Diff}:[\Sigma(\mathbb{C}P^2\times\mathbb{S}^5),G/O]\to L_{10}(e)$ is non-zero.

We note from \cite[Lemma 3.8]{DC10} that $\Theta^{Diff}: [\Sigma (\mathbb{C}P^2\times\mathbb{S}^5), G/O]\to L_{10}(e)$ fits into the following commutative diagram:
\[
\begin{tikzcd}
	{\mathbb{Z}/2\{\nu^2\}\cong\pi_6(G/O)} & {[\Sigma (\mathbb{C}P^2\times\mathbb{S}^5),G/O]} & {[\mathbb{C}P^2\times\mathbb{S}^5\times\mathbb{S}^1,G/O]} \\
	& {L_{10}(e)} & {L_{10}(\mathbb{Z})} \\
	& {\mathcal{S}^{Diff}(\mathbb{C}P^2\times\mathbb{S}^5),}
	\arrow["{Pr^*}", from=1-1, to=1-2]
	\arrow["{T_{\mathbb{C}P^2\times\mathbb{S}^5}}", hook, from=1-2, to=1-3]
	\arrow["{\Theta^{Diff}}"', from=1-2, to=2-2]
	\arrow["{\Theta^{Diff}}", from=1-3, to=2-3]
	\arrow["\cong"', from=2-2, to=2-3]
	\arrow["{\omega^{Diff}}"', from=2-2, to=3-2]
\end{tikzcd}\]
 where the map $T_{\mathbb{C}P^2\times\mathbb{S}^5}:[\Sigma(\mathbb{C}P^2\times\mathbb{S}^5), G/O]\to [\mathbb{C}P^2\times\mathbb{S}^5\times\mathbb{S}^1, G/O]$ is the inclusion onto second factor considering the identification $\mathcal{L}:[\mathbb{C}P^2\times\mathbb{S}^5, G/O]\oplus [\Sigma(\mathbb{C}P^2\times\mathbb{S}^5), G/O]\cong [\mathbb{C}P^2\times\mathbb{S}^5\times\mathbb{S}^1, G/O]$, the induced map $(Pr)^*:\pi_6(G/O)\to [\Sigma(\mathbb{C}P^2\times\mathbb{S}^5),G/O]$ from the projection map $Pr:\Sigma(\mathbb{C}P^2\times\mathbb{S}^5)\simeq \Sigma\mathbb{C}P^2\vee \Sigma \mathbb{S}^5\vee \Sigma^6 \mathbb{C}P^2\to \Sigma \mathbb{S}^5 $ is injective, and $L_{10}(e)\hookrightarrow L_{10}(\mathbb{Z})$ is the isomorphism given in \cite{JS69}. Moreover, for any $y\in [\Sigma(\mathbb{C}P^2\times\mathbb{S}^5), G/O]$, the following diagram commutes:
 \[
 \begin{tikzcd}
	{\mathbb{C}P^2\times\mathbb{S}^5\times\mathbb{S}^1} && {\Sigma(\mathbb{C}P^2\times\mathbb{S}^5)} \\
	& {G/O.}
	\arrow["p", from=1-1, to=1-3]
	\arrow["{\mathcal{L}(0,y)}"', from=1-1, to=2-2]
	\arrow["y", from=1-3, to=2-2]
\end{tikzcd}\]
  Now, from \cite{Wall}, the surgery obstruction of $f\in [\mathbb{C}P^2\times\mathbb{S}^5\times\mathbb{S}^1, G/O]$ is given by
    \begin{align}\label{kervaireobstruction}
    \Theta^{Diff}(f)= &  \langle V(\mathbb{C}P^2\times\mathbb{S}^5\times\mathbb{S}^1)^2  f^*(\bigoplus\limits_{j} \mathcal{K}_{2^{j+1}-2}) , \; [\mathbb{C}P^2\times\mathbb{S}^5\times\mathbb{S}^1] \rangle \nonumber\\
         = &  \langle (1+x^2) f^*(\bigoplus\limits_{j}\mathcal{K}_{2^{j+1}-2}) , \; [\mathbb{C}P^2\times\mathbb{S}^5\times\mathbb{S}^1] \rangle\nonumber \\
          =& \langle x^2 f^*(\mathcal{K}_{6}) , \; [\mathbb{C}P^2\times\mathbb{S}^5\times\mathbb{S}^1] \rangle, 
         \end{align}  
   where $(V(\mathbb{C}P^2\times\s^5\times\s^1))^2= 1+x^2$ is the square of total Wu class of the manifold $\mathbb{C}P^2\times\s^5\times\s^1,x$ is the generator of $H^2(\mathbb{C}P^2;\mathbb{Z}/2)$ and $\bigoplus\limits_{j}\mathcal{K}_{2^{j+1}-2} \in H^{*}(G/O;\mathbb{Z}/2)$ is the smooth Kervaire class. 
   
   We choose $f: \mathbb{C}P^2\times\mathbb{S}^5\times\mathbb{S}^1\to G/O$ such that $$f=\mathcal{L}\circ T_{\mathbb{C}P^2\times\mathbb{S}^5}(Pr^*(\nu^2))=\mathcal{L}(0,Pr^*(\nu^2))=\mathcal{L}(0,\nu^2\circ Pr).$$ 
  Then $f^*= p^*\circ Pr^*(\nu^2)$. Since \( (\nu^2)^*(\mathcal{K}_6) \neq 0 \), and both \( Pr^* \) and \( p^* \) are injective, it follows that $f^*(\mathcal{K}_6) = p^* \circ Pr^* \circ (\nu^2)^*(\mathcal{K}_6) \neq 0.$ Now it follows from (\ref{kervaireobstruction}) that $\Theta^{Diff}(f)\neq 0$ in $L_{10}(e).$ Therefore $\Theta^{Diff}:[\Sigma(\mathbb{C}P^2\times\mathbb{S}^5), G/O]\to L_{10}(e)$ is non-trivial. This completes the proof.
\end{proof}	

We now proceed to give the classification of smooth structures of $\mathbb{C}P^2\times \mathbb{S}^6$. We recall that if $M$ is a closed smooth $n$-manifold and $k\geq 1,$ it is well known that the \textit{relative tangential structure set} \cite{IMRLBW80} $$\mathcal{S}_{k}^{t}(M):= \mathcal{S}^t(M\times\mathbb{D}^k~ \text{rel} ~M\times\s^{k-1})$$
has a group structure \cite{TP70, MR70, WBTP71, bel}. 
Moreover, when $n+k\geq 5$, a relative tangential surgery exact sequence of groups exists \cite[Page-195]{RS87}, which can be expressed as
\begin{equation}\label{relativetangentialsurgery}
   \cdots\to L_{n+k+1}(\pi_1(M))\xrightarrow{\omega_{\rel}^t} \mathcal{S}_{k}^t(M) \xrightarrow{\eta_{\rel}^t}[\Sigma^k M_{+},G]\xrightarrow{\Theta_{\rel}^t} L_{n+k}(\pi_1(M)),
\end{equation}
where $\eta_{\rel}^t: \mathcal{S}_k^t(M)\to [\Sigma^k M_{+},G]$ and $\Theta_{\rel}^t: [\Sigma^k M_{+},G]\to L_{n+k}(\pi_1(M))$ are group homomorphisms.\\
For any closed manifold $M\times \mathbb{D}^k$ of dimension $n+k \geq 5$, the relative topological and PL surgery obstruction maps  
\[
\Theta_{\mathrm{rel}}^{Top} : [ \Sigma^k M_+, G/Top] \to L_{n}(\pi_1(M)), 
\text{ and }~
\Theta_{\mathrm{rel}}^{PL} : [ \Sigma^k M_+, G/PL] \to L_{n+k}(\pi_1(M)),
\]  
of $M\times \mathbb{D}^k$ fit into the commutative diagram
\[\begin{tikzcd}
	{[\Sigma^k M_+, G/PL]} && {L_{n+k}(\pi_1(M))} \\
	{[\Sigma^k M_+, G/Top]} && {L_{n+k}(\pi_1(M)),}
	\arrow["{\Theta_{\rel}^{PL}}", from=1-1, to=1-3]
	\arrow["{\tilde{i}_*}"', from=1-1, to=2-1]
	\arrow["{=}", from=1-3, to=2-3]
	\arrow["{\Theta_{\rel}^{Top}}"', from=2-1, to=2-3]
\end{tikzcd}\]
as given in \cite[p.~289]{kirby} (see also \cite[Page 16]{Surgery_survey}), where $\tilde{i} \colon G/PL \to G/Top$ is the canonical map.

 Using this diagram, it follows from \cite[Page 1639]{DCTM} that, for $k\geq 4$, the commutative diagram
 \begin{equation}\label{commutative_disgram_1}
 \begin{tikzcd}
	{[\Sigma^k M_{+}, G/PL]} && {L_{n+k}(\pi_1(M))} \\
	{[\Sigma (M\times\mathbb{S}^{k-1})_{+},G/PL]} && {L_{n+k}(\pi_1(M))}
	\arrow["{\Theta^{PL}}", from=1-1, to=1-3]
	\arrow["{(\Sigma c)^*}"', from=1-1, to=2-1]
	\arrow["{=}", from=1-3, to=2-3]
	\arrow["{\bar{\Theta}^{PL}}"', from=2-1, to=2-3]
\end{tikzcd}\end{equation} 
also holds, where $c: M\times \mathbb{S}^{k-1}\to \Sigma^{k-1} M_{+}$ is the collapse map.

Furthermore, by \cite[Lemma 3.8]{DC10}, the map $\bar{\Theta}^{PL}: [\Sigma(M\times \mathbb{S}^{k-1})_+, G/PL]$ fits into the following commutative diagram:
\begin{equation}\label{commutative_diagram_2}
\begin{tikzcd}
	{[\Sigma(M\times \mathbb{S}^{k-1})_+,G/PL]} && {L_{n+k}(\pi_1(M))} \\
	{[M\times \mathbb{S}^{k-1}\times \mathbb{S}^1, G/PL]} && {L_{n+k}(\pi_1(M\times \mathbb{S}^1)),}
	\arrow["{\Bar{\Theta}^{PL}}", from=1-1, to=1-3]
	\arrow["{\tilde{T}_{M\times \mathbb{S}^{k-1}}}"', from=1-1, to=2-1]
	\arrow[hook', from=1-3, to=2-3]
	\arrow["{\tilde{\Theta}^{PL}}"', from=2-1, to=2-3]
\end{tikzcd}\end{equation}
where $\tilde{T}_{M\times \mathbb{S}^{k-1}}: [\Sigma(M\times \mathbb{S}^{k-1})_+, G/PL]\to [M\times \mathbb{S}^{k-1}\times \mathbb{S}^1, G/PL] $ is the inclusion map arising from the indentification $[M\times \mathbb{S}^{k-1}\times \mathbb{S}^1,G/PL]\cong [M\times \mathbb{S}^{k-1}, G/PL]\oplus [\Sigma(M\times \mathbb{S}^{k-1}), G/PL]$.

We now revisit some notations. Let $\mathcal{E}_{M}(M\times N)$ denote the subgroup of $\mathcal{E}(M\times N)$ containing $f \in\mathcal{E}(M\times N)$ that fix $M$. In other words, these maps take the form $f=(p_M, p_N\circ f),$ where $p_M$ and $p_N$ are the projection maps from $M\times N$ to $M$ and $N$ respectively. Similarly, we can define the subgroup $\mathcal{E}_N(M\times N)$. We are now going to examine the group of homotopy classes of self-homotopy equivalences of the product manifold $\mathbb{C}P^2\times\mathbb{S}^6$ and establish the following.

\begin{thm}\label{homoequiimplydiffeo}
Let $f\in \mathcal{E}(\mathbb{C}P^2\times\mathbb{S}^6)$. Then $f$ is homotopic to a diffeomorphism.
\end{thm}
To prove Theorem~\ref{homoequiimplydiffeo}, we require the following lemma.
\begin{lemma}\label{FS1diffeo}
    Every self-homotopy equivalence of $\mathbb{C}P^2\times\mathbb{S}^6$ induced by an element of $\pi_6(F_{\mathbb{S}^1}(\mathbb{C}^3))$ is homotopic to a diffeomorphism.
\end{lemma}

\begin{proof}
Let \( f \) be a self-homotopy equivalence of \( \mathbb{C}P^2 \times \mathbb{S}^6 \), induced by an element \( \alpha \in \pi_6(F_{\mathbb{S}^1}(\mathbb{C}^3)) \). Now, to prove this lemma, we need to show that the map \textcolor{blue}{$\alpha$} is mapped trivially under the group homomorphism $\Psi: \pi_6(F_{\mathbb{S}^1}(\mathbb{C}^3))\to \mathcal{S}_6^{Diff}(\mathbb{C}P^2)$. In view of \eqref{relativesmoothsugeryexact}, this is equivalent to showing that the relative normal invariant \( \eta_{\rel}^{Diff}(f) \) is trivial.  

Since $f$ is tangential by Fact~\ref{FS1}(b), we first compute the relative tangential normal invariant $\eta_{\rel}^t(f)$ of $f$. Note from \cite[Page 195]{RS87} that the map \( \eta^t_{\rel} : \mathcal{S}_6^t(\mathbb{C}P^2) \to [\Sigma^6\mathbb{C}P^2, G] \) fits into the following commutative diagram:
    \begin{equation}\label{diagram_tangential}
        \begin{tikzcd}
            {0}\arrow[r]&{[\Sigma^6\mathbb{C}P_{+}^2,PL]}\arrow[r]\arrow[d,"H"']&{[\Sigma^6 \mathbb{C}P_{+}^2,G]}\arrow[r,"\beta"]\arrow[d,-,double equal sign distance,double]&{[\Sigma^6\mathbb{C}P_{+}^2,G/PL]}\arrow[d,"\Theta^{PL}"]\\
          {0}  \arrow[r]&{\mathcal{S}_6^{t} (\mathbb{C}P^2)}\arrow[r,"\eta_{\rel}^t"']& {[\Sigma^6 \mathbb{C}P_{+}^2,G]}\arrow[r,"\Theta_{\rel}^t"']&{L_{10}(e),}
        \end{tikzcd}
    \end{equation} where the first row is the exact sequence induced from the fiber sequence $\cdots\to\Omega (G/PL)\to PL\to G\to G/PL$ and the second row is the exact sequence \eqref{relativetangentialsurgery} with $M=\mathbb{C}P^2$ and $k=6$. Observe that the restricted map \( \beta|_{\pi_6^s} : \pi_6^s \to [\mathbb{S}^6, G/PL] \subset [\Sigma^6 \mathbb{C}P^2_{+}, G/PL] \) is an isomorphism. Next, we show that the restriction $\Theta^{PL}\vert_{\pi_6(G/PL)}: \pi_6(G/PL)\to L_{10}(e)$ is an isomorphism.

   Combining the diagrams \eqref{commutative_disgram_1} and \eqref{commutative_diagram_2}, we see that $\Theta^{PL}: [\Sigma^6\mathbb{C}P^2_+, G/PL]\to L_{10}(e)$ fits into the following commutative diagram
    
    \[\begin{tikzcd}
	{\pi_6(G/PL)} & {[\Sigma^6 \mathbb{C}P^2_+, G/PL]} && {L_{10}(e)} \\
	& {[\mathbb{C}P^2\times\mathbb{S}^5\times\mathbb{S}^1,G/PL]} && {L_{10}(\mathbb{Z}),}
	\arrow["{Pr^*}", from=1-1, to=1-2]
	\arrow["{\Theta^{PL}}", from=1-2, to=1-4]
	\arrow["{T_{\mathbb{C}P^2\times \mathbb{S}^5}\circ (\Sigma c)^*}"', hook, from=1-2, to=2-2]
	\arrow["\cong", from=1-4, to=2-4]
	\arrow["{\tilde{\Theta}^{PL}}"', from=2-2, to=2-4]
\end{tikzcd}\]
  where the projection map $Pr: \Sigma^6 \mathbb{C}P^2_{+}\simeq \Sigma^6 \mathbb{C}P^2\vee \mathbb{S}^6\to \mathbb{S}^6$ induced along $G/PL$ is injective and the isomorphism $L_{10}(e)\hookrightarrow L_{10}(\mathbb{Z})$ follows from \cite{JS69}. Let $g\in [\mathbb{C}P^2\times \mathbb{S}^5\times \mathbb{S}^1, G/PL]$ such that $g= T_{\mathbb{C}P^2\times \mathbb{S}^5}\circ (\Sigma c)^*\circ Pr^*(y)$, where $y$ is the generator of $\pi_6(G/PL)$. Now it follows from \cite{Wall} that 
 \begin{align*}
    \Theta^{Diff}(g)= &  \langle V(\mathbb{C}P^2\times\mathbb{S}^5\times\mathbb{S}^1)^2  g^*(\bigoplus\limits_{j= 0}^{\infty} k_{4j+2}) , \; [\mathbb{C}P^2\times\mathbb{S}^5\times\mathbb{S}^1] \rangle \nonumber\\
         = &  \langle (1+x^2) g^*(\bigoplus\limits_{j=0}^{\infty} k_{4j+2}) , \; [\mathbb{C}P^2\times\mathbb{S}^5\times\mathbb{S}^1] \rangle\nonumber \\
          =& \langle x^2 g^*(k_{6}) , \; [\mathbb{C}P^2\times\mathbb{S}^5\times\mathbb{S}^1] \rangle, 
         \end{align*}where $k_{4j+2}$ is the generator of $H^{4j+2}(G/PL;\mathbb{Z}/2)$. Since $(\Sigma c)^*\circ (Pr)^*\circ (y)^*(k_6)\neq 0$ and $T_{\mathbb{C}P^2\times \mathbb{S}^5}$ is injective, $g^*(k_6)\neq 0$. Therefore $\Theta^{PL}\vert_{\pi_6(G/PL)}\to L_{10}(e)$ is non-trivial.
    
    Then, it follows from the commutative diagram \eqref{diagram_tangential} that the map $\Theta_{\rel}^t \vert_{\pi_6(G)}: \pi_6(G)\to L_{10}(e)$ is non-zero. Therefore, $\eta_{\rel}^t(f)$ is contained in the component $[\Sigma^6\mathbb{C}P^2, G]\cong \mathbb{Z}/3$. We now consider the relative version of the commutative diagram given in \cite[Diagram (2.4)]{IMRLBW80}, as shown below. 
    \[\begin{tikzcd}
	{\mathcal{S}_6^t(\mathbb{C}P^2)} && {[\Sigma^6\mathbb{C}P^2_{+}, G]} \\
	{\mathcal{S}_6^{Diff}(\mathbb{C}P^2)} && {[\Sigma^6\mathbb{C}P^2_{+}, G/O].}
	\arrow["{\eta_{\rel}^t}", from=1-1, to=1-3]
	\arrow[from=1-1, to=2-1]
	\arrow["{j_*}", from=1-3, to=2-3]
	\arrow["{\eta_{\rel}^{Diff}}"', from=2-1, to=2-3]
\end{tikzcd}\]
Since $[\Sigma^6\mathbb{C}P^2, O]$ and $\pi_6(O)$ are both trivial, the map $j_*:[\Sigma^6\mathbb{C}P^2_{+},G]\to [\Sigma^6\mathbb{C}P^2_{+}, G/O]$ is injective. Hence, the above commutative diagram implies that $\eta_{\rel}^{Diff}(f)$ is also contained in the torsion part of $[\Sigma^6\mathbb{C}P^2, G/O]\cong \mathbb{Z}\oplus\mathbb{Z}/3$. Using the fact that $\pi_6(F_{\mathbb{S}^1}(\mathbb{C}^3))$ is finite group \cite[Page-29]{bel}, we deduce that the image of $\alpha$ under the composition $\eta_{\rel}^{Diff}\circ \Psi :\pi_6(F_{\mathbb{S}^1}(\mathbb{C}^3))\to [\Sigma^6\mathbb{C}P_{+}^2,G/O]$ is contained in the $\mathbb{Z}/3$-summand of $[\Sigma^6\mathbb{C}P^2,G/O]$. 

We now consider the following commutative diagram (see \cite[Proposition 2.7]{RS76} \cite[Diagram (3.3)]{RS83}): 
\begin{center}
     \begin{tikzcd}
           {\pi_6(F_{\s^1}(\mathbb{C}^3))}\arrow[rr,"\eta_{\rel}^{Diff}\circ\Psi"]\arrow[d,"s_3"']&&{[\Sigma^6 \mathbb{C}P_{+}^2,G/O]}\\
             {\pi_6(F_{\s^1}(\mathbb{C}^4))\cong \pi_6(F_{\s^1})}\arrow[rr,"\eta_{\rel}^{Diff}\circ\Psi"']&&{[\Sigma^6 \mathbb{C}P_{+}^3,G/O],}\arrow[u,"(\Sigma^6 i' \vee id)^*"']\\
     \end{tikzcd}
     \end{center}
where $s_3:\pi_6(F_{\s^1}(\mathbb{C}^3))\to \pi_6(F_{\s^1})$ is the stabilization map and the right vertical arrow is induced from the inculsion $\Sigma^6 i' :\Sigma^6\mathbb{C}P^2\hookrightarrow \Sigma^6\mathbb{C}P^3$ along $G/O$. Since $\pi_6(F_{\mathbb{S}^1})\cong \mathbb{Z}/2$ \cite{JM82}, the above diagram implies that $\eta_{\rel}^{Diff}\circ \Psi (\alpha)=\eta_{\rel}^{Diff}(f)$ is lies in $\pi_6(G/O)$ summand of $[\Sigma^6\mathbb{C}P^2_{+},G/O]$.   

Combining the above two observations, we conclude that $\eta_{\rel}^{Diff}(f)$ is trivial. This completes the proof.

    \end{proof}

 \begin{proof}[Proof. of Theorem \ref{homoequiimplydiffeo}]
     We begin by observing that for any self-homotopy equivalence $g:\mathbb{C}P^2\times\mathbb{S}^6\to \mathbb{C}P^2\times\mathbb{S}^6,$ the composition map $p_X\circ g\circ i_{X}: X\to X$ is a homotopy equivalence for $X\in \{\mathbb{C}P^2,\mathbb{S}^6\}$, where $i_X: X \hookrightarrow \mathbb{C}P^2\times\mathbb{S}^6$ is the inclusion. Now applying \cite[Theorem 2.5]{PP99}, there is a decomposition
      \begin{equation}\label{selfproduct}
          \mathcal{E}(\mathbb{C}P^2\times\mathbb{S}^6)= \mathcal{E}_{\mathbb{C}P^2}(\mathbb{C}P^2\times\mathbb{S}^6)\cdot \mathcal{E}_{\mathbb{S}^6}(\mathbb{C}P^2\times\mathbb{S}^6).
      \end{equation}
       Further, from \cite[Proposition 2.3]{PP99} $\mathcal{E}_{\mathbb{C}P^2}(\mathbb{C}P^2\times\mathbb{S}^6)$ and $\mathcal{E}_{\mathbb{S}^6}(\mathbb{C}P^2\times\mathbb{S}^6)$ fit into the following two split short exact sequences 
      $$0\to [\mathbb{C}P^2, E_1(\mathbb{S}^6)]\to \mathcal{E}_{\mathbb{C}P^2}(\mathbb{C}P^2\times\mathbb{S}^6)\to \mathcal{E}(\s^6)\to 0$$
      and 
      $$0\to \pi_6(E_1(\mathbb{C}P^2)) \to \mathcal{E}_{\mathbb{S}^6}(\mathbb{C}P^2\times\mathbb{S}^6) \to \mathcal{E}(\mathcal{C}P^2)\to 0,$$ where $[\mathbb{C}P^2,E_1(\mathbb{S}^6)]\cong [\mathbb{C}P^2,SG_7]\cong 0$. We also observe that any self-homotopy equivalence of $\mathbb{C}P^2\times\mathbb{S}^6$ induced from the factor $\mathcal{E}(\mathbb{S}^6)$ or $\mathcal{E}(\mathbb{C}P^2)$ is represented by a diffeomorphism. Furthermore, Lemma \ref{FS1diffeo} assures the same for the self-homotopy equivalences corresponding to the elements of $\pi_6(E_1(\mathbb{C}P^2))\cong \pi_6(F_{\mathbb{S}^1}(\mathbb{C}^3)).$ By combining these in the decomposition (\ref{selfproduct}) and using the composition formula for normal invariants, we deduce that the normal invariant of any self-homotopy equivalence of $\mathbb{C}P^2\times\mathbb{S}^6$ is zero. This concludes the proof of Theorem \ref{homoequiimplydiffeo}. 
\end{proof}
Recall that the homotopy inertia group $I_h(M)$ of a manifold $M,$ is defined as the collection of all $\Sigma \in I(M)$ for which a diffeomorphism $M\#\Sigma \to M$ exists that is homotopic to the standard homeomorphism $h_{\Sigma}: M\#\Sigma \to M.$ Clearly, $I_c(M) \subseteq I_h(M).$ The next lemma confirms that they are indeed equal for the product manifold $\mathbb{C}P^q\times\mathbb{S}^{2k}.$

 \begin{lemma}\label{homotopyinertiaequalconcorinertia}
     For every $q,k \geq 1$, we have $I_h(\mathbb{C}P^q\times\mathbb{S}^{2k})=I_c(\mathbb{C}P^q\times\mathbb{S}^{2k})$.
 \end{lemma}
 \begin{proof}
     This follows quickly from \cite[Corollary 3.2]{RK17}.
 \end{proof}
\begin{theorem}\label{inertiaCP2S6}
    The inertia group of $\mathbb{C}P^2\times\mathbb{S}^6$ is $\mathbb{Z}/2.$
\end{theorem}
\begin{proof}
It follows from Theorem \ref{concorM^4} and Lemma \ref{homotopyinertiaequalconcorinertia} that $I_h(\mathbb{C}P^2\times\mathbb{S}^6)=\mathbb{Z}/2$. Hence, the conclusion immediately obtained from  Theorem \ref{homoequiimplydiffeo}.
\end{proof}

 \begin{theorem}\label{classificationCP2S6}
 If $N$ is a closed, oriented, smooth manifold that is homeomorphic to $\mathbb{C}P^2\times\mathbb{S}^6,$ then $N$ is oriented diffeomorphic to exactly one of the manifolds $\mathbb{C}P^2\times\mathbb{S}^6, (\mathbb{C}P^2\times\mathbb{S}^6)\#\Sigma_{\alpha_1}$, or 
 $(\mathbb{C}P^2\times\mathbb{S}^6)\# (\Sigma_{\alpha_1})^{-1},$ where $\Sigma_{\alpha_1}\in \Theta_{10}$ is the exotic sphere of order $3.$
\end{theorem}
\begin{proof}

Combining the decomposition \eqref{note2.3} with the isomorphism \( [\Sigma^6 \mathbb{C}P^2, Top/O] \cong \mathbb{Z}/3 \) established in Proposition~\ref{prop3.6}, we obtain  
\[
[\mathbb{C}P^2 \times \mathbb{S}^6, Top/O] \cong \mathbb{Z}/3.
\]  
Moreover, the induced degree one map  
\[
f^{*}_{\mathbb{C}P^2 \times \mathbb{S}^6} : \Theta_{10} \to [\mathbb{C}P^2 \times \mathbb{S}^6, Top/O],
\]  
is surjective. This implies that the concordance structure set \( \mathcal{C}(\mathbb{C}P^2 \times \mathbb{S}^6) \) is given by  
\[
\mathcal{C}(\mathbb{C}P^2 \times \mathbb{S}^6) = \{ \left((\mathbb{C}P^2 \times \mathbb{S}^6) \# \Sigma, h_{\Sigma}\right) \mid \Sigma \in \mathbb{Z}/3 \subset \Theta_{10} \},
\]
where $h_{\Sigma}: (\mathbb{C}P^2\times\mathbb{S}^6)\#\Sigma\to \mathbb{C}P^2\times \mathbb{S}^6$ is the canonical homeomorphism corresponding to the element $\Sigma \in \Theta_{10}$. It now follows from Theorem~\ref{inertiaCP2S6} that neither \( (\mathbb{C}P^2 \times \mathbb{S}^6) \# \Sigma_{\alpha_1} \) nor \( (\mathbb{C}P^2 \times \mathbb{S}^6) \# \Sigma_{\alpha_1}^{-1} \) is diffeomorphic to \( \mathbb{C}P^2 \times \mathbb{S}^6 \), where \( \Sigma_{\alpha_1} \) denotes the exotic $10$-sphere of order 3.

Moreover, we claim that \( (\mathbb{C}P^2 \times \mathbb{S}^6) \# \Sigma_{\alpha_1} \) is not (oriented) diffeomorphic to \( (\mathbb{C}P^2 \times \mathbb{S}^6) \# \Sigma_{\alpha_1}^{-1} \). Indeed, if they were oriented diffeomorphic, it would follow that \( 2\Sigma_{\alpha_1} \in I(\mathbb{C}P^2 \times \mathbb{S}^6) = \mathbb{Z}/2 \). However, this contradicts the fact that \( \Sigma_{\alpha_1} \in \Theta_{10} \) has order $3$.


This completes the proof.
\end{proof}

Moreover, the following theorem provides the corresponding classification for manifolds tangentially homotopy equivalent to $\mathbb{C}P^{2} \times \mathbb{S}^{6}$.

  \begin{theorem}\label{tangetial_classification}
      
     If a closed, oriented, smooth manifold \( N \) is tangentially homotopy equivalent to \( \mathbb{C}P^2 \times \mathbb{S}^6 \), then \( N \) is oriented diffeomorphic to one of the following non-diffeomorphic manifolds: \( \mathbb{C}P^2 \times \mathbb{S}^6 \), \( (\mathbb{C}P^2 \times \mathbb{S}^6) \# \Sigma_{\alpha_1} \), or \( (\mathbb{C}P^2 \times \mathbb{S}^6) \# \Sigma_{\alpha_1}^{-1} \), where \( \Sigma_{\alpha_1} \) denotes the exotic 10-sphere of order $3$.
\end{theorem}  
\begin{proof}
   Let $f:N\to \mathbb{C}P^2\times\mathbb{S}^6$ be a tangential homotopy equivalence. 

If $\eta^{Diff}(f)$ is trivial, then it follows from the smooth surgery exact sequence that $N$ is diffeomorphic to $\mathbb{C}P^2\times\mathbb{S}^6$.

  Suppose $\eta^{Diff}(f)$ is non-trivial. Then we prove that the normal invariant $\eta^{Diff}(f)$ lies in the torsion part of the summand $[\Sigma^6\mathbb{C}P^2, G/O]$ of $[\mathbb{C}P^2\times\mathbb{S}^6, G/O]$. To prove this, We note from \cite[Page 107]{DCIH15} that the tangential structure set $\mathcal{S}^t(\mathbb{C}P^2\times \mathbb{S}^6)$ for $\mathbb{C}P^2\times \mathbb{S}^6$ fits into the following commutative diagram
  \[\begin{tikzcd}
	0 & {\mathcal{S}^t(\mathbb{C}P^2\times\mathbb{S}^6)} & {[\mathbb{C}P^2\times \mathbb{S}^6, G]\cong \mathbb{Z}/3\oplus \pi_6^s} & {L_{10}(e)} \\
	0 & {\mathcal{S}^{Top}(\mathbb{C}P^2\times\mathbb{S}^6)} & {[\mathbb{C}P^2\times \mathbb{S}^6, G/Top]} & {L_{10}(e),}
	\arrow[from=1-1, to=1-2]
	\arrow["{=}"', from=1-1, to=2-1]
	\arrow["{\eta^{t}}", from=1-2, to=1-3]
	\arrow[from=1-2, to=2-2]
	\arrow["{\Theta^t}", from=1-3, to=1-4]
	\arrow["{\mathcal{J}_*}"', from=1-3, to=2-3]
	\arrow["{=}", from=1-4, to=2-4]
	\arrow[from=2-1, to=2-2]
	\arrow["{\eta^{Top}}"', from=2-2, to=2-3]
	\arrow["{\Theta^{Top}}"', from=2-3, to=2-4]
\end{tikzcd}\]
where $\mathcal{J}: G\to G/Top$ is the canonical map. Since the topological surgery obstruction map $\Theta^{Top}:\pi_6(G/Top)\to L_6(e)$ for $\mathbb{S}^6$ is an isomorphism \cite[Page 421]{Luck_surgery}, we obtain from \cite[Proposition 3.1]{DC10} and \cite[Page 288]{AR92} that the restriction of the map $\Theta^{Top}: [\mathbb{C}P^2\times\mathbb{S}^6, G/Top]\to L_{10}(e)$ to the summand $\pi_{6}(G/Top)$ is non-trivial. Moreover, since $\pi_6(Top)$ is trivial, the restriction $\mathcal{J}_*\vert_{\pi_6^s}: \pi_6^s\to [\mathbb{C}P^2\times\mathbb{S}^6, G/Top]$ is injective. Therefore, by the commutativity of the diagram, it follows that $\Theta^t\vert_{\pi_6^s}: \pi_6^s\to L_{10}(e)$ is non-trivial. Hence, $\eta^t(f)\in \mathbb{Z}/3\cong [\Sigma^6\mathbb{C}P^2, G]\subset [\mathbb{C}P^2\times\mathbb{S}^6, G]$. To compute the normal invarinat of $f$, we consider the following commutative digram from \cite[Page 107]{DCIH15}
 \[\begin{tikzcd}
	0 & {\mathcal{S}^t(\mathbb{C}P^2\times\mathbb{S}^6)} & {[\mathbb{C}P^2\times \mathbb{S}^6, G]\cong \mathbb{Z}/3\oplus \pi_6^s} & {L_{10}(e)} \\
	0 & {\mathcal{S}^{Diff}(\mathbb{C}P^2\times\mathbb{S}^6)} & {[\mathbb{C}P^2\times \mathbb{S}^6, G/O]} & {L_{10}(e).}
	\arrow[from=1-1, to=1-2]
	\arrow["{=}"', from=1-1, to=2-1]
	\arrow["{\eta^{t}}", from=1-2, to=1-3]
	\arrow[from=1-2, to=2-2]
	\arrow["{\Theta^t}", from=1-3, to=1-4]
	\arrow["{j_*}"', from=1-3, to=2-3]
	\arrow["{=}", from=1-4, to=2-4]
	\arrow[from=2-1, to=2-2]
	\arrow["{\eta^{Diff}}"', from=2-2, to=2-3]
	\arrow["{\Theta^{Diff}}"', from=2-3, to=2-4]
\end{tikzcd}\]
Since $[\Sigma^6\mathbb{C}P^2, O], \pi_6(O)$ and $[\mathbb{C}P^2, O]$ are all trivial, it follows from the spilitting \eqref{note2.3} and the long exact sequence induced from $O\xrightarrow{\Omega \widetilde{J}}G\xrightarrow{j}G/O$ that $j_*:[\mathbb{C}P^2\times\mathbb{S}^6, G]\to [\mathbb{C}P^2\times\mathbb{S}^6, G/O]$ is injective. Hence, from the above commutative diagram, we get that $\eta^{Diff}(f)$ lies in the torsion summand of $[\Sigma^6\mathbb{C}P^2, G/O]\cong \mathbb{Z}\oplus\mathbb{Z}/3$. Therefore, from the smooth surgery exact sequence for \( \mathbb{C}P^2 \times \mathbb{S}^6 \), it follows that the element \( (N, f) \in \mathcal{S}^{Diff}(\mathbb{C}P^2 \times \mathbb{S}^6) \) is equivalent to either  
\[
(\mathbb{C}P^2 \times \mathbb{S}^6 \# \Sigma_{\alpha_1}, h_{\Sigma_{\alpha_1}}) \quad \text{or} \quad (\mathbb{C}P^2 \times \mathbb{S}^6 \# \Sigma_{\alpha_1}^{-1}, h_{\Sigma_{\alpha_1}^{-1}}),
\]  
where both elements are non-standard in \( \mathcal{S}^{Diff}(\mathbb{C}P^2 \times \mathbb{S}^6) \) by Theorem~\ref{inertiaCP2S6}, and they represent distinct classes by Theorem~\ref{classificationCP2S6}. This completes the proof.  
\end{proof}


\end{document}